\documentclass{amsart}

\usepackage{amsmath,amsthm,amssymb,array,enumerate,parskip,graphicx,tabularx}
\usepackage[all]{xy}

\usepackage{xcolor} 
\usepackage{hyperref}
\definecolor{dark-red}{rgb}{0.5,0.15,0.15}
\definecolor{dark-blue}{rgb}{0.15,0.15,0.6}
\definecolor{dark-green}{rgb}{0.15,0.6,0.15}
\hypersetup{
    colorlinks, linkcolor=dark-red,
    citecolor=dark-blue, urlcolor=dark-green
}

\usepackage{comment}

\newcommand{\Z}{\mathbb{Z}} 
\newcommand{\Q}{\mathbb{Q}} 
\newcommand{\C}{\mathbb{C}} 
\newcommand{\F}{\mathbb{F}} 
\newcommand{\Fpbar}{\bar{\mathbb{F}}_p} 
\newcommand{\G}{\mathbb{G}}

\newcommand{\ee}{e}
\newcommand{\ff}{f}

\newcommand{\A}{\mathbb{A}} 

\newcommand{\Gal}[1]{G_{#1}} 
\newcommand{\rGal}[2]{\mathrm{Gal}(#1/#2)} 

\newcommand{\e}{\zeta} 

\newcommand{\dsh}{d^{\prime sh}} 
\newcommand{\LL}{{\rm Ker}( \dsh )} 

\newcommand \CO {{\mathcal O}}
\newcommand \BA {\mathbb{A}}
\newcommand \PP {{\mathbb P}^1}
\newcommand \ZZ {{\mathbb Z}}
\newcommand \CC {{\mathbb C}}
\newcommand \QQ {{\mathbb Q}}
\newcommand \NN {{\mathbb N}}
\newcommand  \FF {{\mathbb F}}
\newcommand \Ker {\mathop {\rm Ker}}

\DeclareMathOperator{\dlog}{dlog}
\DeclareMathOperator{\Nm}{Nm}
\DeclareMathOperator{\im}{im}

\newtheorem{theorem}{Theorem}[section]

\newtheorem{lemma}[theorem]{Lemma} 
\newtheorem{corollary}[theorem]{Corollary}

\newtheorem{proposition}[theorem]{Proposition} \theoremstyle{definition}
\newtheorem{remark}[theorem]{Remark} 
\newtheorem{definition}[theorem]{Definition}

\newtheorem{fact}[theorem]{Fact}

\newcommand{\Aa}{A} 
\newcommand{\Ash}{A^{sh}} 

\title{Galois action on the homology of Fermat curves}
\author{Rachel Davis}
\address{Purdue University}
\email{davis705@math.purdue.edu}
\author{Rachel Pries}
\address{Colorado State University}
\email{pries@math.colostate.edu}
\author{Vesna Stojanoska}
\address{Max Planck Institute for Mathematics}
\email{vstojanoska@mpim-bonn.mpg.de}
\author{Kirsten Wickelgren}
\address{Georgia Institute of Technology}
\email{kwickelgren3@math.gatech.edu}

\thanks{We would like to thank BIRS for hosting the WIN3 conference where we began this project.  Some of
this work was done while the third and fourth authors were in
residence at MSRI during the spring 2014 Algebraic topology semester, supported by NSF
grant 0932078 000.
The second author was supported by NSF grant DMS-1101712. The third author was supported by NSF grant DMS-1307390. The fourth author was supported by an American Institute of Mathematics five year fellowship and NSF grant DMS-1406380. 
We would also like to thank Sharifi and the referee for helpful remarks.}

\begin{document}

\begin{abstract}
In \cite{Anderson}, Anderson determines the homology of the degree $n$ Fermat curve as a Galois module 
for the action of the absolute Galois group $G_{\QQ(\zeta_n)}$. 
In particular, when $n$ is an odd prime $p$, he shows that the action of $G_{\QQ(\zeta_p)}$ 
on a more powerful relative homology group factors through the Galois group of the splitting field of the polynomial $1-(1-x^p)^p$.
If $p$ satisfies Vandiver's conjecture, we give a proof that the Galois group $G$ 
of this splitting field over $\QQ(\zeta_p)$ is an elementary abelian $p$-group of rank $(p+1)/2$.
Using an explicit basis for $G$, we completely compute the relative homology, the homology, and the homology of an open subset of the degree $3$ Fermat curve as Galois modules. We then compute several Galois cohomology groups which arise in connection with obstructions to rational points.
 \end{abstract}

\maketitle

\section{Introduction}

The Galois actions on the \'etale homology, cohomology, and homotopy groups of varieties carry information about rational points. We revisit results of Anderson \cite{Anderson} on a relative homology group of the Fermat curve of prime exponent to make his results amenable to computations of groups such as $H^1(G_S, \pi_1^{\rm ab})$ and $H^2(G_S, \pi_1^{\rm ab} \wedge \pi_1^{\rm ab})$ where $G_S$ denotes a Galois group of a maximal extension of a number field with restricted ramification and $\pi_1^{\rm ab}$ denotes the abelianized geometric fundamental group of the Fermat curve, or of an open subset. These groups arise in obstructions of Ellenberg to rational points \cite{Ellenberg_2_nil_quot_pi} as well as in McCallum's application of the method of Coleman and Chabauty to Fermat curves \cite{McCallum94}. 

Let $k$ be a number field. 
The Fermat curve of exponent $n$ is the smooth projective curve $X \subset {\mathbb P}_k^2$ of genus $g=(n-1)(n-2)/2$ given by the equation
\[x^n + y^n = z^n.\] The affine open $U \subset X$ given by $z \not = 0$ has affine equation $x^n+y^n=1$.
The closed subscheme $Y \subset X$ defined by $xy=0$ consists of $2n$ points. Let $H_1(U,Y; \ZZ/n)$ denote the \'etale homology group of the pair $(U \otimes \overline{k},Y \otimes \overline{k})$, which is a continuous module over the absolute Galois group $\Gal{k}$ of $k$. The $\mu_n \times \mu_n$ action on $X$ given by $$(\e^i, \e^j) \cdot [x,y,z] = [\e^i x,\e^j y,z] , ~~~~~~(\e^i, \e^j) \in \mu_n \times \mu_n $$ determines an action on $U$ and $Y$.  These actions give $H_1(U,Y; \ZZ/n)$ the structure of a $(\ZZ/n)[\mu_n \times \mu_n]$ module. As a $(\ZZ/n)[\mu_n \times \mu_n]$ module, $H_1(U,Y; \ZZ/n)$ 
is free of rank one \cite[Theorem 6]{Anderson}, with generator denoted $\beta$.    
It follows that the Galois action of $\sigma \in \Gal{k}$ is determined by $\sigma \beta = B_{\sigma} \beta$ for some $B_\sigma \in  (\ZZ/n)[\mu_n \times \mu_n]$.

Anderson shows that $B_{\sigma}$ is determined by an analogue of the classical gamma function $\Gamma_{\sigma} \in \ZZ/n^{\rm sh}[\mu_n]$,
where $ \ZZ/n^{\rm sh}$ denotes the strict Henselization of $\ZZ/n$. 
In particular, there is a formula \cite[Theorem 9, Theorem 7]{Anderson} recalled in \eqref{EdbarBsigma} as the equation $\dsh(\Gamma_{\sigma}) = B_{\sigma}$ with $\dsh$ defined in \eqref{d'def} and immediately below. 
The canonical derivation $d: \ZZ/n^{\rm sh}[\mu_n] \to \Omega \ZZ/n^{\rm sh}[\mu_n]$ from the ring $\ZZ/n^{\rm sh}[\mu_n]$ to its module of K\"ahler differentials allows one to take the logarithmic derivative $\dlog \Gamma_{\sigma}$ of $\Gamma_{\sigma}$, which it is convenient to view as an element of a particular quotient of $\Omega \ZZ/n^{\rm sh}[\mu_n]$. See Section \ref{Ssurvey}. For $n$ prime, $\dlog \Gamma_{\sigma}$ determines $B_{\sigma}$ uniquely \cite[10.5.2,10.5.3]{Anderson}. The function $\sigma \mapsto \dlog \Gamma_{\sigma}$ is in turn determined by a relative homology group of the punctured affine line $H_1(\mathbb{A}^1 - V(\sum_{i=0}^{n-1} x^i), \{ 0, 1\}; \ZZ/n)$ \cite[Theorem 10]{Anderson}. 
Putting this together, Anderson shows that, for $n=p$ a prime, the $\Gal{\QQ(\zeta_p)}$ action on $H_1(U,Y; \ZZ/p)$ 
factors through $\rGal{L}{\QQ(\zeta_p)}$ where $L$ is the splitting field of $1-(1-x^p)^p$.
Ihara \cite{Ihara} and Coleman \cite{Coleman89} obtain similar results from different viewpoints. 

Let $K$ denote the cyclotomic field $K=\QQ(\e_n)$, where $\e_n$ denotes a primitive $n$th root of unity, and let $\Gal{K}$ be its absolute Galois group. 
When the exponent is clear, let $\e$ denote $\e_n$ or $\e_p$ for a prime $p$.
Let $\kappa$ denote the classical Kummer map; for $\theta \in K^*$, let 
$\kappa(\theta): \Gal{K} \to \mu_n$ be defined by
\[\kappa (\theta) (\sigma) = \frac{\sigma \sqrt[n]{\theta}}{ \sqrt[n]{\theta}}.\] 
In Proposition \ref{Pkappac}, we determine $\dlog \Gamma_{\sigma}$ in terms of the classical Kummer map for all $n \geq 3$, modulo indeterminacy which does not affect $B_{\sigma}$, with the answer being 
$\dlog \Gamma_{\sigma} = \sum_{i=1}^{n-1} \kappa(1-\zeta^{-i})(\sigma) \e^i \dlog \e$.

Recall that Vandiver's Conjecture for a prime $p$ is that $p$ does not divide $h^+$, where $h^+$ is the order of the 
class group of $\QQ(\zeta_p+\zeta_p^{-1})$.  It has been verified for all $p$ less than 163 million. 
For $n=p$ a prime satisfying Vandiver's conjecture, we give a proof that $\rGal{L}{K}$ is isomorphic to $(\ZZ/p)^r$ with $r = (p+1)/2$ in Proposition \ref{Pgrouprank}. This is false for $p$ not satisfying Vandiver's conjecture as seen in
Remark \ref{p_not_Vandiver_GalLK_remark}.
There are a couple of natural choices for such an isomorphism. 
In Corollary \ref{C_iso}, we show that the following map gives an isomorphism: $$\Phi=\kappa(\e) \times \prod_{i = 1}^{\frac{p-1}{2}}\kappa(1-\e^{-i}) : {\rm Gal}(L/K) \to (\mu_p)^{\frac{p+1}{2}}.$$

For $p=3$, we use the formula for $\dlog \Gamma_\sigma$ to compute $B_{\sigma}$ explicitly in Lemma \ref{p=3formulaB}. It is possible to extend this calculation to compute $B_{\sigma}$ for all primes $p$ and we will make this computation available in a forthcoming paper. 
(As seen in Remark \ref{n_non_prime_Kerdlog_remark}, the element $\dlog \Gamma_{\sigma}$ and \cite[10.5.2]{Anderson} do not determine $B_{\sigma}$ when $n$
is not prime so the calculation of $B_{\sigma}$ when $n$ is not prime will require further input.)
Combining the above, we obtain:

\begin{theorem}\label{thm:action3}
Let $p=3$ and $K = \QQ(\e_p)$.
The $\Gal{K}$-action on $H_1(U,Y; \ZZ/p)$ factors through $\Gal{K} \to \rGal{L}{K}$, 
where $L$ denotes the splitting field of $1-(1-x^p)^p$ (or equivalently of $x^6-3x^3+3$). 
Write $H_1(U,Y; \ZZ/p) \cong \ZZ_p[\e_0,\e_1]/\langle \e_0^p -1, \e_1^p-1\rangle$
and $\rGal{L}{K} \cong \ZZ/p \times \ZZ/p$.
Then $(c_0, c_1) \in \ZZ/p \times \ZZ/p$ acts on $\ZZ_p[\e_0,\e_1]/\langle \e_0^p -1, \e_1^p-1\rangle$ 
by multiplication by $B_{\sigma} = \sum_{i,j=0}^{p-1} b_{i,j} \e_0^i \e_1^j$ where
\begin{equation*}
\begin{aligned}
b_{0,0} &= 1+ c_0 - c_0^2\\
b_{0,1} &= c_1-c_0^2 \\
b_{1,1} &= -c_1-c_0^2.
\end{aligned}
\end{equation*}
and where the rest of the coefficients $b_{i,j}$ are determined by $b_{i,j}=b_{j,i}$, and the fact that $b_{0,0}+b_{0,1}+b_{0,2} = 1$, $b_{1,0}+b_{1,1}+b_{1,2} = 0$, and $b_{2,0}+b_{2,1}+b_{2,2} = 0$.
\end{theorem}

We have an analogous calculation of $H_1(U,Y; \ZZ/p)$ for all primes $p$ satisfying Vandiver's conjecture, which we will make available shortly.

Given the Galois action on $H_1(U,Y; \ZZ/n)$, we compute the Galois actions on $H_1(U; \ZZ/n)$ and $H_1(X; \ZZ/n)$ for all $n \geq 3$
in Section \ref{Shomnonrel}.

These computations can be used to study rational points on varieties in the following way. Let $Z$ be a scheme over $k$, and for simplicity assume that $Z$ has a rational point $b$. (This assumption is unnecessary, but it is satisfied in the situations encountered in this paper and it simplifies the exposition.) Choose a geometric point of $Z$ with image $b$ and let $\pi = \pi_1(Z_{\overline{k}},b)$ denote the geometric \'etale fundamental group of $Z$ based at the chosen geometric point. The generalized {\em Kummer map} associated to $Z$ and $b$ is the map $\kappa: Z(k) \to H^1(\Gal{k}, \pi)$ defined by $$\kappa (x) = [\sigma \mapsto \gamma^{-1} \sigma \gamma] $$ where $\gamma$ is an \'etale path from $b$ to a geometric point above $x$. 
Before returning to the potential application to rational points, we remark that the map $\kappa$ is functorial and the computation of  $\dlog \Gamma_{\sigma}$ in Proposition \ref{Pkappac} is obtained by applying $\kappa$ to the $K$-map $\mathbb{A}^1 - V(\sum_{i=0}^{n-1} x^i) \to \G_m^{n-1}$. 

From $\kappa$, we also obtain a map $\kappa^{\rm ab,p}: Z(k) \to H^1(\Gal{k}, \pi^{\rm ab} \otimes \ZZ_{p})$ defined to be the composition of $\kappa$ with the map $H^1(\Gal{k}, \pi) \to H^1(\Gal{k}, \pi^{\rm ab}\otimes \ZZ_{p})$ induced by the quotient map $\pi \to \pi^{\rm ab}\otimes \ZZ_{p}$, where $\ZZ_p$ denotes the $p$-adic integers. For $Z$ a curve or abelian variety over a number field,  $\kappa^{\rm ab,p}$ is well-known to be injective. Let $S$ denote a set of places of $k$ including the infinite places, all the primes of bad reduction of $Z$ and a place above $p$. Let $G_S= \pi_1(\mathcal{O}_k[1/S])$ denote the Galois group of the maximal extension of $k$ ramified only over $S$. Assume that $Z$ is proper to simplify exposition. Then $\kappa^{\rm ab,p}$ factors through a map $\kappa^{\rm ab,p}: Z(k) \to H^1(G_S, \pi^{\rm ab} \otimes \ZZ_{p})$. Let $\pi=[\pi]_1 \supseteq [\pi]_2 \supseteq \ldots$ denote the lower central series of the profinite group $\pi$, 
where $[\pi]_m$ is the closure of the subgroup $[[\pi]_{m-1}, \pi]$
generated by commutators of elements of $\pi$ with elements of $[\pi]_{m-1}$. 
Using work of Schmidt and Wingberg \cite{SW92}, Ellenberg \cite{Ellenberg_2_nil_quot_pi} defines a series of obstructions to a point of the Jacobian of a curve $Z$ lying in the image of the Abel-Jacobi map associated to $b$. The first of these obstructions is defined using a map $$\delta_2:  H^1(G_S, \pi^{\rm ab} \otimes \ZZ_{p}) \to H^2(G_S, ([\pi]_2/[\pi]_3) \otimes \ZZ_{p})$$ such that $\Ker \delta_2 \supset Z(k)$. Zarkhin defines a similar map \cite{zarhin}. The group $([\pi]_2/[\pi]_3) \otimes \ZZ_{p}$ fits into a short exact sequence $$ 0 \to \ZZ_p(1) \to (\pi^{\rm ab} \wedge \pi^{\rm ab}) \otimes \ZZ_{p}\to ([\pi]_2/[\pi]_3) \otimes \ZZ_{p} \to 0.$$ There are mod $p$ versions of $\delta_2$ and the generalized Kummer maps. A more detailed account of Ellenberg's obstructions is in \cite{Wickelgren3nil}.

Thus computations of $H^1(G_S, \pi^{\rm ab} \otimes \ZZ/p)$ and $H^2(G_S, (\pi^{\rm ab} \wedge \pi^{\rm ab}) \otimes \ZZ/p)$ give information about rational points. Groups closely related to $H^1(G_S, \pi^{\rm ab} \otimes \ZZ/p)$ also appear in \cite{C-NGJ} and \cite{McCallum94}. 

The final section of this paper includes calculations of $H^1({\rm Gal}(L/K), M)$ and $H^2({\rm Gal}(L/K), M)$ for $M$ each of 
$H_1(U, Y; \ZZ/n)$, $H_1(U; \ZZ/n)$, $H_1(X; \ZZ/n)$, and $H_1(U, \ZZ/n) \wedge H_1(U, \ZZ/n)$. These can be inserted into the Hochschild-Serre spectral sequence $$H^i(\rGal{L}{K}, H^j( G_{S,L} ,M)) \Rightarrow H^{i+j}(G_{S,K}, M),$$ where $G_{S,L}$ denotes the Galois group of the maximal extension of $L$ only ramified at places above $S$, and $G_{S,K} = G_S$. Since $H_1(U, Y; \ZZ/n)$, $H_1(U; \ZZ/n)$, $H_1(X; \ZZ/n)$ are $\pi^{\rm ab} \otimes \ZZ/n$ for $Z = U/Y$, $Z=U$ and $Z=X$ respectively, these are groups mentioned above, and appear in Ellenberg's obstructions. This is the subject of on-going work. 

\subsection{Notation}

Let $n \geq 3$ be an integer; often $n$ will be a prime $p$.
Let $\e$ be a fixed primitive $n$-th root of unity and $K = \QQ(\e)$.
For brevity, let $\Aa = \ZZ/n$, and let $\Ash$ denote the strict Henselization of $\Aa$. 
If $n=p$ is prime, then the field $\Aa$ is a Henselian local ring and its strict Henselization is the separable closure
$\Ash \simeq \Fpbar$. 

If $k$ is any number field, $\Gal{k}$ denotes the absolute Galois group of $k$. 

\begin{definition}\label{Kummer_map_def} 
Given a primitive $n$th root $\sqrt[n]{\theta}$ of $\theta \in k$ and $\sigma \in \Gal{k}$, then $\kappa(\theta) \sigma$ is the element of $A$ 
such that 
\[\sigma \sqrt[n]{\theta} = \zeta^{\kappa(\theta) \sigma} \sqrt[n]{\theta}.\]
\end{definition}

\begin{remark}\label{Kummer=generalized_Kummer_Gm}
The map $\kappa: k^* \to H^1(\Gal{k}, \ZZ/n(1))$ defined by letting $\kappa(\theta)$ be represented by the twisted homomorphism $\sigma \mapsto \kappa(\theta) \sigma$ is the generalized Kummer map of $\G_{m,k}$ with base point $1 \in \G_{m,k}(k)$. Here $\ZZ/n(1)$ is the Galois module with underlying group $\ZZ/n$ and Galois action given by the cyclotomic character. See, for example, \cite[12.2.1 Example 1]{Wickelgren3nil}.
\end{remark}

For $\theta \in K^*$ and $n=p$, the map $\kappa(\theta): \Gal{K} \to \ZZ/p$ is a homomorphism and is independent of the choice of $p$th root of $\theta$ because $\mu_p \subset K$.

\section{Anderson's results, revisited} \label{Ssurvey}

In this section, we recall results from \cite{Anderson} that are relevant for this paper. 
Recall that $K=\QQ(\e)$, that $U \subset \A^2_K$ denotes the affine Fermat curve over $K$ with equation
$x^n + y^n =1$,
and that $Y \subset U$ is the divisor defined by $xy=0$. 
The path $\beta: [0,1] \to U(\CC)$ given by $t \mapsto (\sqrt[n]{t},   \sqrt[n]{1-t})$, where $\sqrt[n]{-}$ denotes the real $n$th root, determines a singular $1$-simplex in the homology of $U$ relative to $Y$ whose class we denote by the same name.

For $m \in \NN$, let $\Lambda_m$ denote the group ring over $\Aa$ of the finite group $\mu_n(\C)^{\times (m+1)}$. 
Then $\Lambda_m$ has a natural $\Gal{K}$-action.
For $0 \leq i \leq m$, let $\e_i$ denote a primitive $n$th root of unity in the $i$th copy of $\mu_n(\C)$.  Then 
\[  \Lambda_m = \Aa [\e_0, \dots, \e_m ]/ (\e_0^n -1 , \dots, \e_m^n - 1). \]

There is an action of $\Lambda_1$ on $U$ given by $\zeta^i_0 \times \zeta^j_1: (x,y) \mapsto (\zeta^i_0 x, \zeta^j_1y)$.
This action stabilizes $Y$. Thus the relative homology group $H_1(U,Y; \Aa)$ is a $\Lambda_1$-module.
Note that $H_1(U,Y; \Aa)$ has rank $n^2$ over $A$. 

Anderson describes the $\Gal{K}$-action on $H_1(U,Y; \Aa)$.
First, \cite[Theorem 6]{Anderson} states that $H_1(U,Y; \Aa)$ is a free rank one module over $\Lambda_1$ 
generated by the class $\beta$. 

Specifically, $\sigma \in \Gal{K}$ acts $\Aa$-linearly, and
\[ \sigma \cdot ( \e_0^i \e_1^j \beta) = (\sigma\cdot \e_0^i) (\sigma \cdot \e_1^j) B_\sigma \beta, \]
where $B_\sigma $ is a unit in $\Lambda_1$ defined by
\[ \sigma \cdot \beta = B_{\sigma} \beta.\]
Thus to describe the $\Gal{K}$-action on $H_1(U,Y; \Aa)$, it is necessary and sufficient to describe the action on the element $\beta$.

Anderson also proves that the action of the absolute Galois group $\Gal{\Q}$ on $H_1(U, Y; \Aa)$ 
factors through a finite quotient. This result is a consequence of the analysis in the rest of the section.  
In particular, if $n$ is a prime $p$, then $\sigma \in \Gal{K}$ acts trivially on $H_1(U,Y; \Aa)$ if and only if 
$\sigma$ fixes the splitting field $L$ of the polynomial $f_p=1-(1-x^p)^p$, \cite[Section 10.5]{Anderson}.
In Section \ref{sec:GalLK}, we prove that 
$\rGal{L}{K}$ is an elementary abelian $p$-group of rank at most $(p+1)/2$.

Anderson highlights the following application of this result.
By \cite[Lemma, page 558]{Anderson}, there is a connection between the 
action of $\sigma \in \Gal{\Q}$ on $H_1(U,Y; \Aa)$ and the action of $\sigma$ on the fields of definition of points of a generalized Jacobian of $X$.

\begin{theorem}[\cite{Anderson}, Theorem 0] \label{L=coord_pt_genJac}
Let $S$ be the generalized Jacobian of $X$ with conductor $\infty$.
Let $b$ denote the $\Q$-rational point of $S$ corresponding to the difference of the points $(0,1)$ and $(1,0)$. 
The number field generated by the coordinates of the $n$th roots of $b$ in $S(\overline{\QQ})$ contains 
the splitting field $L$ of the polynomial $1-(1-x^n)^n$, with equality if $n$ is prime.
\end{theorem}

Information on fields generated by points of the Jacobian of quotients of Fermat curves is also contained in \cite{C-NGJ},  \cite{Coleman98},  \cite{Greenberg}, and\cite{Tzermias}.

In the remainder of this section, we describe Anderson's method for determining $B_\sigma$. 
Let $b_{i,j}$ denote the coefficients of $B_\sigma$, so that 
\[B_\sigma= \sum_{0\leq i,j < n} b_{i,j} \e_0^i \e_1^j.\]
It will often be convenient to arrange the coefficients of $B_\sigma$ in an $n \times n$ matrix.

Let $w:\Lambda_1 \to \Lambda_1$ be the map induced by swapping the two copies of $\mu_n(\C)$, i.e. by swapping $\e_0$ and $\e_1$. Then $w$ preserves the units in $\Lambda_1$.
Let $(\Lambda_1^\times)^w$ denote the symmetric units, i.e., the units fixed by $w$. 
If $a_{i,j}\in \Aa$, then an element 
\[ \sum_{0\leq i,j < n} a_{i,j} \e_0^i \e_1^j \in \Lambda_1^\times \]
is in $(\Lambda_1^\times)^w$ precisely when $a_{i,j} = a_{j,i}$ for all $i,j$. 

\begin{fact}\label{fact:symmetry}
\cite[Theorem 7]{Anderson}
If $\sigma \in \Gal{\Q}$, then $B_{\sigma} \in (\Lambda_1^\times)^w$.
In other words, 
the coefficients of $B_\sigma$ are symmetric; $b_{i,j}=b_{j,i}$ for any $0\leq i,j <n$.
\end{fact}

Next, consider the map $d^{\prime\prime}:(\Lambda_1^\times)^w \to \Lambda_2^\times$ given by
\[ \sum a_{i,j} \e_0^i \e_1^j \mapsto \frac{\left( \sum a_{i,j} \e_0^j \e_1^i \e_2^j \right) \left( \sum a_{i,j} \e_0^i \e_2^j \right) }{ \left( \sum a_{i,j} \e_0^i \e_1^j \e_2^j \right) \left( \sum a_{i,j} \e_1^i \e_2^j \right)}. \]
By \cite[Theorem 7]{Anderson}, $B_\sigma$ is in the kernel of $d^{\prime\prime}$. In particular, there is an equality in 
$\Lambda_2^\times$, given by
\[\left( \sum b_{i,j} \e_0^j \e_1^i \e_2^j \right) \left( \sum b_{i,j} \e_0^i \e_2^j \right) = \left( \sum b_{i,j} \e_0^i \e_1^j \e_2^j \right) \left( \sum b_{i,j} \e_1^i \e_2^j \right).\]
This gives, via the map $\Lambda_2^\times \to \Lambda_1^\times$ sending $\e_2 \mapsto 1$, the equality
\[\left( \sum b_{i,j} \e_0^j \e_1^i \right) \left( \sum b_{i,j} \e_0^i  \right) = \left( \sum b_{i,j} \e_0^i \e_1^j \right) \left( \sum b_{i,j} \e_1^i\right). \]
By Fact \ref{fact:symmetry}, the first terms on each side cancel giving
\[ \sum b_{i,j} \e_0^i  = \sum b_{i,j} \e_1^j. \]
This is only possible if the following is true.
\begin{fact}\label{fact:zerosum} \cite[10.5.4]{Anderson}
If $1 \leq i \leq n$, then
$ \sum_{0\leq j <n} b_{i,j} =0$.
\end{fact}

In other words, the entries of each column of the matrix $B_\sigma$ sum up to zero, for all but the 
zeroth column.  By Fact \ref{fact:symmetry}, the entries of each row of the matrix $B_\sigma$ also sum up to zero, 
for all but the zeroth row.  

Furthermore, consider the map $d^\prime:\Lambda_0^\times \to (\Lambda_1^\times)^w$ given by
\begin{equation}\label{d'def}\sum a_{i} \e_0^i \mapsto \frac{\left( \sum a_i \e_0^i\right) \left( \sum a_i \e_1^i\right)}{\left( \sum a_i \e_0^i \e_1^i\right)}, \end{equation}
as well as its extension $\dsh : \bar\Lambda_0^\times \to (\bar\Lambda_1^\times)^w$, where 
$\bar\Lambda_i = \Lambda_i \otimes_{\Aa} \Ash$. 
The kernel $\LL$  is determined in \cite[Proposition 8.3.1]{Anderson}; 
when $n=p$ is prime, it is the cyclic subgroup of order $p$ multiplicatively generated by $\e_0$. 

\begin{fact}\label{fact:Gammaunique} \cite[Theorem 9]{Anderson}
Let $n \geq 3$ and let $\LL$ denote the kernel of $\dsh$.
In $\bar\Lambda_0^\times / \LL$, 
there exists a unique element $\Gamma_\sigma$ which maps to $B_\sigma$ under $\dsh$. 
\end{fact}
In the sequel, the notation $\Gamma_\sigma$ will also be used to denote an element of $\bar \Lambda_0^\times$ representing this coset in $\bar\Lambda_0^\times / \LL$.

\begin{fact} \label{Faugmentation}
\cite[9.6 and 10.5.2]{Anderson} 
The difference $B_\sigma - 1$ lies in the augmentation ideal $(1-\epsilon_0)(1-\epsilon_1) \Lambda_1$. 
\end{fact}

Consider the element $\Gamma_\sigma$ such that 
\begin{equation} \label{EdbarBsigma}
\dsh (\Gamma_\sigma) = B_\sigma.
\end{equation}
By Fact \ref{fact:Gammaunique}, in order to determine $B_\sigma$, it suffices to find the preimage $\Gamma_\sigma$. 
To accomplish this, Anderson looks at the logarithmic derivative homomorphisms from the groups of units $\Lambda_k^\times$ to the K\"ahler differentials $\Omega(\Lambda_k)$. This has the geometric meaning of comparing with ``the circular motive," where the Galois action is more transparent.

There is a commutative square
\[\xymatrix{ 
\bar\Lambda_0^\times \ar[r]^{d^\prime} \ar[d]_{\dlog} & (\bar \Lambda_1^\times)^w \ar[d]^{\dlog} \\
\Omega(\bar\Lambda_0) \ar[r] &\Omega(\bar \Lambda_1)^w,
} \]
where the bottom horizontal map is defined analogously to $d^\prime$. Note that for each $m$, the $\bar\Lambda_m$-module $\Omega(\bar\Lambda_m)$ is free on generators $\{\dlog \e_i\}_{0\leq i \leq m}$. 

Here is some notation needed to describe $\dlog \Gamma_\sigma$.
Let $\tilde{V} = \BA^1 - \mu_n$ and let $V=\tilde{V} \cup \{1\}$.  
Let $\lambda_0$ be a small counterclockwise loop around $1$.
Choose the isomorphism \[H_1(\tilde{V};A) = A[\mu_n] \lambda_0 \simeq \Omega A[\Lambda_0],\]
where $\lambda_0 \mapsto \frac{d\epsilon_0}{\epsilon_0}$.

Consider the exact sequence from \cite[\S 9]{Anderson}
\[0 \to A \lambda_0 \to H_1(\tilde{V}; A) \to H_1(V; A) \to 0, \]
or 
\begin{equation}\label{eq:sesH1V}0 \to A \frac{d\epsilon_0}{\epsilon_0} \to \Omega A[\Lambda_0] \to H_1(V; A) \to 0,\end{equation}
which identifies $H_1(V; \Aa) $ as a quotient of $\Omega(\Lambda_0)$. 

Let $Z$ denote the subscheme of $V$ defined by the vanishing of $x_0(1-x_0)$, i.e., the points $0$ and $1$ in $V$. Let $\psi \in H_1(V,Z; \Aa)$ denote the homology class represented by the cycle given by the interval $[0,1]$. 
Let $(\sigma-1)\psi$ denote the cycle given by concatenating the path $\sigma \psi$ and the path $\psi$ traveled in reverse. Since $\Gal{K}$ fixes the endpoints of $\psi$, the cycle $(\sigma-1)\psi$ represents a class in 
$H_1(V; \Aa) = H_1(V, \emptyset; \Aa)$. 

Let $\Psi_\sigma$ denote the coset in $\Omega(\Lambda_0) / \Aa \dlog \e_0$ which corresponds to the homology class 
of $(\sigma-1)\psi$ under \eqref{eq:sesH1V}.
The following theorem computes $\dlog\Gamma_{\sigma}$ to be $\Psi_\sigma$.

\begin{theorem}\label{fact:dlogPsi} \cite[Theorem 10]{Anderson}
$\dlog\Gamma_\sigma \in \Omega(\bar \Lambda_0)$ represents the  $\Ash \dlog \e_0$-coset $\Psi_\sigma$.
\end{theorem}

For this paper, the importance of Theorem \ref{fact:dlogPsi} lies in the geometric description of $\Psi_\sigma$. 
This description shows that $\sigma \mapsto \Psi_\sigma$ is the image of a rational point under a generalized Kummer map of the sort which arises in the section conjecture. We use this observation to compute $\Psi_\sigma$ in Section \ref{S_psi}. By Theorem \ref{fact:dlogPsi}, we have therefore also computed $\dlog\Gamma_\sigma$.

To give a complete description of $H_1(U,Y;\Aa)$ as a Galois module, it thus suffices to
achieve the following goal,
which we complete in Section \ref{Sexplicit} for the case $n=3$ and in future work for $n$ an odd prime.

{\bf Goal: reconstruct $B_\sigma$ from $\Psi_\sigma$.}

\section{Galois group of the splitting field of $1-(1-x^p)^p$ over $K$} \label{sec:GalLK}

Let $n=p$ be an odd prime and let $\e$ be a primitive $p$th root of unity.
The choice of $\e$ fixes an identification $\ZZ/p \to \mu_p$ by sending $i$ to $\e^i$.
Let $K=\QQ(\e)$.

Let $L$ be the splitting field of the polynomial $f_p(x)=1-(1-x^p)^p$. 
In Proposition~\ref{Pgrouprank}, we determine the structure of the Galois group $G=\rGal{L}{K}$
for primes $p$ satisfying Vandiver's conjecture.  
The techniques in this section are well-known to experts but we could not find an off-the-shelf reference 
for this result.  Before starting the proof, we describe some motivation for it in the next remark.  

\begin{remark}
\begin{enumerate}
\item As seen in Theorem \ref{L=coord_pt_genJac}, $L$ is the field of definition of the $p$th roots of a point $b$ in 
a certain generalized Jacobian.
By \cite[Section 10.5]{Anderson}, an automorphism $\sigma \in \Gal{K}$ 
acts trivially on $H_1(U, Y; \Aa)$ if and only if $\sigma \in \Gal{L}$. 
In view of this result, to determine the action of $\Gal{K}$ on $H_1(U, Y; \Aa)$, it remains to determine the action of the 
finite Galois group $\rGal{L}{K}$. 

\item 
We would like to thank the referee for pointing out related work in \cite{Greenberg}.
Recall that the Jacobian of the Fermat curve $X$ of exponent $p$ is isogenous to 
${\mathbb J}=\prod_{a=1}^{p-2} J_a$ where $J_a$ is the Jacobian of the curve $y^p=x^a(1-x)$.
Consider the field extension $L_{\mathbb J}$ of $\QQ$ generated by the points of order $p$ on ${\mathbb J}$.
In \cite[Theorem 4]{Greenberg}, Greenberg proves that $L_{\mathbb J}$ is the field  
$K(\{\sqrt[p]{\eta} \mid \eta \in C^+)$ generated over $K$ by
the $p$th roots of the real cyclotomic units.
(Note that Lemma \ref{Lequalfields} below implies that $L_{\mathbb J} \subset L$.)
He remarks that ${\rm Gal}(L_{\mathbb J}/K) \simeq (\ZZ/p)^t$ with $t \leq (p-3)/2$ and that $t=(p-3)/2$ when $p$ satisfies 
Vandiver's conjecture.

\item 
We would like to thank Sharifi for pointing out similar work in \cite[Section 2.8]{AndersonIhara}, where 
the authors determine the Galois group of the Galois closure of 
$\sqrt[p]{1-\sqrt[p]{1-\zeta}}$ over $K$.  That extension is non-abelian over $\QQ(\mu_{p^2})$, in contrast with 
the extension in this paper which is abelian even over $K$.
\end{enumerate}
\end{remark}

\subsection{The splitting field of $1-(1-x^p)^p$}\

The prime $p$ is totally ramified in $K$ with $p=\langle 1-\e \rangle^{p-1}$ \cite[Lemma 1.4]{washingtonbook}.
Thus there is a unique place $\nu = \langle 1-\e \rangle $ above $p$ in $K$. 
Also, $p = \prod_{i = 1}^{p-1} (1- \e^i)$ and $(1- \e^i)/(1-\e)$ is a unit of $\CO_K$ by \cite[Lemma 1.3]{washingtonbook}.
Thus $\nu = \langle 1 - \e^i \rangle$ for all $i = 1,2,\ldots, p-1$.
Since $1=v_{\nu} (p) = (p-1) v_{\nu} (1- \e)$, it follows that $v_{\nu} (1- \e^i) = 1/(p-1)$ for $1 \leq i \leq p-1$. 

Let $L'$ be the maximal elementary abelian $p$-group extension of $K$ unramified except over $\nu=\langle 1-\e \rangle$.  

\begin{lemma} \label{Lelemab}
\begin{enumerate}
\item $L=K(\sqrt[p]{1-\e^i}, 1 \leq i \leq p-1)$.
\item $L \subset L'$ and ${\rm Gal}(L/K)$ is an elementary abelian $p$-group. 
\end{enumerate}
\end{lemma}

\begin{proof}
\begin{enumerate}
\item Let $z=1-x^p$ where $x$ is a root of $1-(1-x^p)^p$.  The equality $z^p=1$ implies that $K \subset L$.
The $p^2-p$ non-zero roots of $f_p(x)$ are the $p$th roots of $1-\e^i$ for $1 \leq i \leq p-1$.
Thus $L=K(\sqrt[p]{1-\e^i}, 1 \leq i \leq p-1)$.

\item The field $L$ is the compositum of the fields $K(\sqrt[p]{1-\e^i})$.  
For each $i$, the extension $K(\sqrt[p]{1-\e^i})/K$ is a Galois degree $p$ extension ramified only above $1-\e^i$ and $\infty$.
This proves both statements.
\end{enumerate}
\end{proof}

\begin{lemma} \label{Lequalfields}
The field $L$ is the same as the fields $L_2$ and $L_3$ where
\[L_2=K(\sqrt[p]{1-\e^i}, 1 \leq i \leq \frac{p-1}{2}, \sqrt[p]{p});\]
\[L_3=K(\sqrt[p]{1-\e^i}, 1 \leq i \leq \frac{p-1}{2}, \sqrt[p]{\e}).\]
\end{lemma}

\begin{proof}  The idea of the proof is to show $L \subseteq L_3 \subseteq L_2 \subseteq L$. 

{\bf $L \subseteq  L_3$}: 
For $\frac{p-1}{2} < i \leq p-1$, write $j=-i$. 
Then 
\[\sqrt[p]{1-\e^j} = \sqrt[p]{1-\e^{-i}} = \sqrt[p]{\e^{-i}-1} \cdot \sqrt[p]{-1}.\]
Since $p$ is odd, $\sqrt[p]{-1} \in K$.
So 
\[\sqrt[p]{1-\e^j} = \sqrt[p]{\e^{-i}(1-\e^i)} \cdot \sqrt[p]{-1}
= \sqrt[p]{1-\e^i} \cdot (\sqrt[p]{\e})^{-i} \cdot \sqrt[p]{-1}  \in L_3.\]

{\bf $L_3 \subseteq L_2$}:  

Let $\e_{p^2}$ denote a $p$th root of $\e$. It suffices to show that $\e_{p^2} \in L_2$. 
Write $p=bc$ with 
\[b = \prod_{i=1}^{\frac{p-1}{2}} (1- \e^i), \ c= \prod_{i=\frac{p+1}{2} }^{p-1} (1- \e^i).\]
Note that $(1-\e^i)/(1-\e^{-i})=-\e^i$. Thus, $\frac{b}{c}=(-1)^{\frac{p-1}{2}} \e^{\frac{(p-1)(p+1)}{8}}$
and
\[b^2=\frac{b}{c} \cdot bc = (-1)^{\frac{p-1}{2}} \e^{\frac{(p-1)(p+1)}{8} }\cdot p.\]
Then \[ \e^{\frac{(p-1)(p+1)}{8} } = (-1)^{\frac{p-1}{2} } p^{-1} \prod_{i=1}^{\frac{p-1}{2}} (1- \e^i)^2. \] 

Let $J= (p-1)^2(p+1)/16$ and note that $p \nmid J$.  Raising both sides of the previous equation to the power $\frac{p-1}{2}\frac{1}{p}$ shows that 
\[\e_{p^2}^J= \zeta' (\sqrt[p]{-1})^{ \frac{(p-1)^2}{4} }   (\sqrt[p]{p })^{\frac{1-p}{2} }  \displaystyle \prod_{i=1}^{\frac{p-1}{2} }  \left(\sqrt[p]{(1-\e^i)^2 } \right)^{\frac{p-1}{2}}, \]
for some $p$th root of unity $\zeta'$.  Thus $\e_{p^2}^J \in L_2$ and $\e_{p^2} \in L_2$.

{\bf $L_2 \subseteq L$}: 
This follows from the equality $\sqrt[p]{p} = \prod_{i=1}^{p-1} \sqrt[p]{1-\e^i}$.
\end{proof}

\subsection{Background on units in cyclotomic fields}\

Let $K=\QQ(\e)$ and let $K^+ = \QQ(\e+\e^{-1})$.  
Let $E=\CO_K^*$ (resp.\ $E^+=\CO_{K^+}^*$) denote the group of units in $\CO_K$ (resp.\ $\CO_{K^+}$).
Let $V$ denote the subgroup of $K^*$ generated by $\{\pm \e, 1-\e^{i} : i = 1,2, \ldots, p-1 \}$.
Let $W$ be the group of roots of unity in $K$.

Consider the cyclotomic units  $C = V \cap \CO^*$ of $K$ and
the cyclotomic units $C^+ = C \cap (\CO^+)^*$ of $K^+$ \cite[page 143]{washingtonbook}. 
By \cite[Lemma 8.1]{washingtonbook}, $C$ is generated by $\e$ and $C^+$;
and $C^+$ is generated by $-1$ and the units
\[\epsilon_a = \e^{(1-a)/2} (1-\e^a)/(1-\e), \] for $1 < a < p/2$.
By \cite[Theorem 4.12]{washingtonbook}, the index of $WE^+$ in $E$ is $1$ or $2$.
Let $h^+$ denote the order of the class group of $K^+$.

\begin{theorem}  \label{TindexV} \cite[Theorem 8.2]{washingtonbook}
The index of the cyclotomic units $C^+$ in $E^+$ is the class number $h^+$ of $K^+$.  
Thus if Vandiver's Conjecture is true for the prime $p$, then $E/E^p$ is generated by $C$.
\end{theorem}

\begin{remark}
Vandiver's Conjecture (first conjectured by Kummer in 1849) states that $p$ does not divide the class number $h^+$.  
It has been verified for all $p$ less than $163$ million \cite{BHvandiver}.
It is also true for all regular primes.
\end{remark}

\subsection{The Galois group of $1-(1-x^p)^p$}

\begin{proposition} \label{Pgrouprank}
If Vandiver's Conjecture is true for the prime $p$, then
the Galois group of $L/K$ is an elementary abelian $p$-group of rank $(p+1)/2$.
\end{proposition}

\begin{proof}
By Lemma \ref{Lelemab}, ${\rm Gal}(L/K)$ is an elementary abelian $p$-group.
Let $r$ be the integer such that ${\rm Gal}(L/K) \simeq (\Z/p)^r$.
The field $L$ is obtained by adjoining $p$th roots of elements in some subgroup $B \subset K^*/(K^*)^p$, and by Kummer theory $B \simeq (\Z/p)^r$. By Lemma \ref{Lequalfields}, $B$ is generated by $\e$ and $1-\e^i$ for $1 \leq i \leq (p-1)/2$.
Thus $r \leq (p+1)/2$. Thus it suffices to show that $r \geq (p+1)/2$.

Note that $B$ is generated by $\e$ and $1-\e^i$ for $1 \leq i \leq (p-1)/2$.  
Thus $B$ is also generated by $\e$, $1-\e$, and $\epsilon_a$ for $1 < a < p/2$. 
Consider the subgroup $B'$ of $K^*/(K^*)^p$ generated by $\e$ and $\epsilon_a$ for $1 < a < p/2$. 
Let $r'$ be the rank of $B'$ over $\Z/p$.
Since $\e$ and $\epsilon_a$ are units, and $1-\e$ has positive valuation at the prime above $p$, 
it suffices to show that $r' \geq (p-1)/2$.

Since $-1$ is a $p$th power, $B'$ is also the subgroup generated by the cyclotomic units $C$.  By hypothesis, $p$ satisfies Vandiver's conjecture and so Theorem \ref{TindexV} implies that $B' \simeq E/E^p$. By Dirichlet's unit theorem, $E\simeq \ZZ^{\frac{p-1}{2}-1} \times \mu_p$. Thus $r'= \frac{p-1}{2}-1 + 1 = (p-1)/2.$
\end{proof}

We now describe an explicit set of generators for ${\rm Gal}(L/K)$.
Given a primitive $p$th root $\sqrt[p]{\theta}$ of $\theta \in K$ and $\sigma \in \Gal{K}$, 
recall from Definition \ref{Kummer_map_def} that $\kappa(\theta) \sigma$ is the element of $\ZZ/p$ 
such that 
\[\sigma \sqrt[p]{\theta} = \zeta^{\kappa(\theta) \sigma} \sqrt[p]{\theta}.\]

\begin{corollary}\label{C_iso}
Let $p$ be an odd prime such that $p \nmid h^+$. 
Then the following map is an isomorphism: 
\[\Phi=\kappa(\e) \times \prod_{i = 1}^{\frac{p-1}{2}}\kappa(1-\e^{-i}) : {\rm Gal}(L/K) \to (\ZZ/p)^{\frac{p+1}{2}}.\] 
\end{corollary}

\begin{proof}
By Lemma \ref{Lequalfields}, $L = K(\sqrt[p]{\e}, \sqrt[p]{1-\e^{-i}}:  i = 1,2, \ldots, \frac{p-1}{2})$.  
Let $G \subseteq K^*/(K^*)^p$ denote the subgroup
generated by $S=\{ \e, 1-\e^{-i} : i = 1,2, \ldots, \frac{p-1}{2}\}$. 
By Kummer theory, it suffices to show that $S$ 
is a $\ZZ/p$-basis for the $\ZZ/p$-vector space ${\rm Gal}(L/K)$, which follows from Proposition \ref{Pgrouprank}.
\end{proof}

\begin{remark}\label{p_not_Vandiver_GalLK_remark}
If $p \mid h^+$ then $p$ divides $[E^+ : C^+]$ by \cite[Theorem 8.2]{washingtonbook}. Since $E^+$ does not contain the $p$th roots of unity, $E^+$ has no $p$-torsion, and it follows that there is an element $c$ of $C^+$ which is a $p$th power of an element in $E^+$, but not a $p$th power of any element of $C^+$. Since $-1$ is a $p$th power and $C^+$ is generated by $-1$ and  $\{ \epsilon_a : 1 < i < p/2\}$, $c$ may be taken to be $c = \prod_{a=2}^{(p-1)/2} \epsilon_a^{e_a}$ with $0 \leq e_a \leq p-1$. Since $B'$ in the previous proof is generated by $C$, 
it follows that $B'$ is generated by $\{\e, \epsilon_a : 1 < i < p/2\}$. Since $c$ maps to $0$ in $E/E^p$, this implies that the rank $r'$ of $B'$
is less than the cardinality of  $\{\e, \epsilon_a : 1 < i < p/2\}$.
Thus $r= r'+ 1 <(p+1)/2$. Thus if Vandiver's Conjecture is not true for the prime $p$, then the rank of the elementary abelian $p$-group $\rGal{L}{K}$ is strictly less than $(p+1)/2$.
\end{remark}

\section{Comparison with an $(n-1)$-torus} \label{Skappa}\label{S_psi}

Recall the notation from Section \ref{Ssurvey} that $\tilde{V} = \BA^1 - \mu_n$, $V=\tilde{V} \cup \{1\}$, and $Z$ consists of the points $0$ and $1$ in $V$. Recall that $\psi \in H_1(V,Z; \Aa)$ denotes the homology class represented by the path from $0$ to $1$ along the real axis, and that $\Psi_\sigma$ is defined to be the element of $\Omega(\Lambda_0) / \Aa \dlog \e_0$ determined by $(\sigma-1)\psi$ and the exact sequence $$0 \to \Aa \dlog \e_0 \to \Omega(\Lambda_0) \to H_1(V;\Aa) \to 0,$$ where the quotient map $ \Omega(\Lambda_0) \to H_1(V;\Aa)$ is the map of $\Lambda_0$ modules mapping $\e_0 \dlog \e_0$ to a small counterclockwise loop around $\e$. 

Note that there is a map from $V$ to a torus which induces a Galois equivariant isomorphism on $H_1(-; \Aa)$. For example, this map could be the Abel-Jacobi map to the generalized Jacobian. Furthermore, over $K$, this torus splits, and it is easy to write down a map to a split torus inducing an isomorphism on $H_1(-; \Aa)$. Namely, the map
\[ f: V_K \to (\G_{m,K})^{\times n-1} ,\] given by $z \mapsto (z-\e, z-\e^2,\dots, z-\e^{n-1})$ induces a Galois equivariant isomorphism on $H_1(-; \Aa)$. 

In this section, we use the isomorphism $H_1(f; \Aa)$ to compute $\Psi_\sigma$ in terms of the classical Kummer map, 
relying on the facts that $H_1((\G_{m,K})^{\times (n-1)}; \Aa) \cong \Aa^{n-1}$ and that the map $\kappa$ for $\G_m$ can be identified with the classical Kummer map.
We will furthermore see in Section \ref{Compatible} that this computation is compatible with Section \ref{sec:GalLK}.

\subsection{Computation of $\Psi_\sigma$}  \label{Scomputepsi}

Fix the isomorphism $I:\Omega(\Lambda_0) / \Aa \dlog \e_0 \to \Aa^{n-1}$ given by $$\sum_{i=1}^{n-1} a_i \e_0^i \dlog \e_0 \mapsto (a_1, a_2, \ldots, a_{n-1}).$$
This isomorphism $I$ can also be obtained by composing the isomorphism described above
$\Omega(\Lambda_0) / \Aa \dlog \e_0 \cong H_1(V;\Aa)$ with $H_1(f;\Aa)$ and an obvious isomorphism $H_1((\G_{m, K})^{\times (n-1)}; \Aa) \cong \Aa^{n-1}$. 

\begin{proposition} \label{Pkappac}
With notation as above,
\[\Psi_\sigma=(\kappa(1-\zeta^{-1})(\sigma), \ldots, \kappa(1-\zeta^{-(n-1)})(\sigma)).\]
\end{proposition}

\begin{proof}
Consider the maps $\kappa^{ab}_{V,b}: V(K) \to H^1(\Gal{K}, H_1(V))$, defined so that $\kappa^{ab}_{V,b}(x)$ is represented by the cocycle $$\sigma \mapsto \gamma^{-1} \sigma \gamma$$ where $\gamma$ is a path from $b$ to $x$, and composition of paths is written from right to left, so $\gamma^{-1} \sigma \gamma$ is a loop based at $b$.
As in \cite[p. 8]{Wickelgren3nil},
the dependency on the choice of basepoint $b$ in $V$ is \begin{equation}\label{kappa_change_bp}\kappa_{b^\prime}(x)  = \kappa_b(x) -  \kappa_{b}(b^\prime).\end{equation} 
By definition, $\Psi_\sigma$ is the element of $H_1(V; A)$ determined by $(\sigma-1)\psi$. Note that $(\sigma-1)\psi = \kappa_{V,0}^{{\rm ab}}(1)(\sigma)$.

Since $\kappa$ is functorial, one sees that $H_1(f) (\sigma -1) \psi = \kappa_{T,f(0)}^{{\rm ab}}(f(1))(\sigma)$, where $T$ is the torus $T = (\G_{m,K})^{\times (n-1)}$ and $f$ is the map $V_K \to T$ defined above. 

Since the geometric fundamental group respects products over algebraically closed fields of characteristic $0$ \cite[XIII Proposition 4.6]{sga1}, the map $\kappa_T= \kappa^{\rm ab}_T$ for $T$ decomposes as the product of the maps $\kappa$ for $\G_{m,K}$ which are each given by $\kappa_{\G_{m,K},1}(\theta) (\sigma) = \kappa(\theta) \sigma$ as in Definition \ref{Kummer_map_def} and Remark \ref{Kummer=generalized_Kummer_Gm}. Thus $ \kappa_{T,f(0)}^{{\rm ab}}(f(1))(\sigma)$ is identified with $\prod_{i=1}^{n-1} \kappa_{\G_{m,K},-\e^i}(1-\e^i) (\sigma)$ when, via the projection maps, $\pi_1(T_{\overline{k}},1)$ is identified with $\prod_{i=1}^{n-1} \pi_1(\G_{m,\overline{k}},1)$.

Applying \eqref{kappa_change_bp} with $b=1$, using the fact that $\kappa$ from Definition \ref{Kummer_map_def} is a homomorphism, yields that $\prod_{i=1}^n \kappa_{\G_{m,K},-\e^i}(1-\e^i) (\sigma) = \prod_{i=1}^n \kappa(\frac{1-\e^i}{-\e^i})(\sigma)$. 
The proposition follows from the above, since $(1-\e^i)/(-\e^{i}) = 1 - \e^{-i}$.
\end{proof}

Combining with Theorem \ref{fact:dlogPsi} (c.f. \cite[Theorem 10]{Anderson}), we obtain:

\begin{corollary}\label{cor:GammaPsi}
Modulo a term of the form $\alpha \dlog\e $, with $\alpha \in \Ash $,
\[\dlog (\Gamma_{\sigma}) = \sum_{i=1}^{n-1} c_i \e^i \dlog \e, \ {\rm with} \ c_i=\kappa(1-\zeta^{-i})(\sigma).\]
\end{corollary}

\subsection{Compatibility with Section \ref{sec:GalLK}} \label{Compatible}

\begin{remark}
In computing $\kappa(1 - \zeta^{-i})(\sigma)$ for $k=K$, one can restrict to the image $\overline{\sigma} \in {\rm Gal}(L/K)$.
\end{remark}

\begin{corollary} \label{Cformulapsin=p}
Suppose $n=p$ is a prime satisfying Vandiver's conjecture.
With respect to the isomorphism $\Phi: {\rm Gal}(L/K) \to \Aa^{\frac{p+1}{2}}$ from Corollary \ref{C_iso} 
and the isomorphism $I:\Omega(\Lambda_0) / \Aa \dlog \e_0 \to \Aa^{p-1}$ from Section \ref{Scomputepsi},
the map \[{\rm Gal}(L/K) \to \Omega(\Lambda_0) / \Aa \dlog \e_0, \  \sigma \mapsto \Psi_{\sigma}\] is the 
explicit $\Aa$-linear map 
\[(c_0, c_1, \ldots, c_{\frac{p-1}{2}}) \mapsto (c_1, c_2, \ldots, c_{\frac{p-1}{2}}, c_{\frac{p-1}{2}} +  \scriptstyle{\frac{p-1}{2} } \displaystyle{  c_0, 
\ldots , c_2 + 2 c_0, c_1 + c_0 ).}\]
\end{corollary}

\begin{proof}
By Proposition \ref{Pkappac}, $\Psi_{\sigma}$ is computed 
\[\Psi_\sigma=(\kappa(1-\zeta^{-1})(\sigma), \ldots, \kappa(1-\zeta^{-(p-1)})(\sigma))\] with respect to the isomorphism $I$.
For $i = 1,2, \ldots, \frac{p-1}{2}$, then $\kappa(1- \zeta^{-i})$ 
is identified with the projection onto $c_i$, the $(i+1)$st coordinate of $(\ZZ/p)^{\frac{p+1}{2}} \cong {\rm Gal}(L/K)$ via the isomorphism $\Phi$.
Recall that $(1- \zeta^i)/(1-\zeta^{-i}) = -\zeta^i$ and $-1$ is a $p$th power since $p$ is odd.
Thus \[\kappa(1-\e^i) - \kappa(1-\e^{-i}) = i \kappa(\e) = ic_0.\] Rearranging terms yields that $ \kappa(1-\e^{-i})  = \kappa(1-\e^i) - i c_0$.
Applying this equation when $i = \frac{p-1}{2}+1, \ldots, p-1$ shows that $\kappa(1-\e^{-i}) = \kappa(1-\e^{-(p-i)}) - i c_0 = c_{p-i} - i c_0$.  This
implies that $\kappa(1- \zeta^{-i})$ is the projection onto the $(p-i+1)$st coordinate $c_{p-i}$ plus $p-i$ times the projection onto the first coordinate $c_0$.
\end{proof}

\subsection{Coordinate sum of $\Psi_{\sigma}$}

We include the following result for its own interest; it is not needed in the computation of $H_1(U,Y;A)$, $H_1(U;A)$, or $H_1(X;A)$ as Galois modules, and it is not needed in the computations of Section \ref{Scohom}.
For $\sigma \in {\rm Gal}(L/K)$, write $\Psi_\sigma=(c_1, \ldots, c_{p-1})$ as in Corollary \ref{Cformulapsin=p}.

\begin{lemma} \label{Lsumc}
If $\sigma \in G_M$, with $M=\QQ(\sqrt[p]{p})$, then $\sum_{i=1}^{p-1} c_i \equiv 0 \bmod p$.
More generally, if $M_1=\QQ(\zeta_p, \sqrt[p]{p})$ and if
$\tau \in {\rm Gal}(M_1, \QQ(\zeta_p))$ is such that $\tau(\sqrt[p]{p})=\zeta_p^j \sqrt[p]{p}$, 
then $\sum_{i=1}^{p-1} c_i \equiv j \bmod p$.
\end{lemma}

\begin{proof}
Write $\theta_i=1-\zeta_p^{-i}$ and note that $\prod_{i=1}^{p-1} \theta_i = p$.
Thus $\prod_{i=1}^{p-1} \sqrt[p]{\theta_i} = \sqrt[p]{p} \in M$ is fixed by $\sigma \in G_M$.
So $\prod_{i=1}^{p-1} \sigma(\sqrt[p]{\theta_i}) = \sqrt[p]{p}$.
By definition, $\sigma(\sqrt[p]{\theta_i}) = \zeta_p^{\kappa_p(\theta_i) \sigma} \sqrt[p]{\theta_i}$.
By Proposition \ref{Pkappac}, $c_i=\kappa_p(\theta_i) \sigma$.
Thus, 
\[\sqrt[p]{p}=\prod_{i=1}^{p-1} \zeta_p^{c_i} \sqrt[p]{\theta_i} = \zeta_p^{\sum_{i=1}^{p-1} c_i} \sqrt[p]{p}.\]
It follows that $\sum_{i=1}^{p-1} c_i \equiv 0 \bmod p$.

Similarly, 
\[\zeta_p^j \sqrt[p]{p}= \tau(\sqrt[p]{p})=\prod_{i=1}^{p-1} \zeta_p^{c_{i,\tau}} \sqrt[p]{\theta_i} 
= \zeta_p^{\sum_{i=1}^{p-1} c_{i,\tau}} \sqrt[p]{p},\]
so $\sum_{i=1}^{p-1} c_{i, \tau} \equiv j \bmod p$.
\end{proof}

\section{Explicit computation of $B_\sigma$} \label{Sexplicit}

\subsection{Determining $B_\sigma$ from $\Psi_\sigma$}

Recall from Fact \ref{fact:Gammaunique} that $\Gamma_\sigma $ is an element of $\bar \Lambda_0^\times$, unique modulo the kernel of $\dsh: \bar \Lambda_0^\times \to \Lambda_1^\times$, such that 
\[\dsh (\Gamma_\sigma) = B_\sigma.\] 
Corollary \ref{cor:GammaPsi} determines the coefficients of the logarithmic derivative $\dlog \Gamma_\sigma$; they are the ones appearing in $\Psi_\sigma$, and explicitly described in Proposition \ref{Pkappac}.

When $n$ is prime, the kernel of $\dlog$ is easy to manage
and thus $\Psi_\sigma$ determines the action of $\Gal{K}$ on $H_1(U,Y;A)$
as seen in the next result.
This result is implicit in \cite[10.5]{Anderson}.

\begin{proposition}\label{prop:BfromPsi}
Let $n=p$ be a prime. Then $ \Psi_\sigma$ uniquely determines $B_\sigma$.
\end{proposition}

The following lemmas will be useful for the proof of Proposition \ref{prop:BfromPsi}.

\begin{lemma}\label{lemma:kerdlog}
The kernel of $\dlog: \bar\Lambda_0^\times \to \Omega(\bar\Lambda_0)$ consists of elements $x=\sum_{0\leq i <n} a_i \e_0^i$ such that $ia_i =0 \in \Aa$ for all $0\leq i < n$. In particular, when $n$ is prime, the kernel of $\dlog$ consists of the constant (in $\e_0$) invertible polynomials $(\Ash)^\times \subset \bar\Lambda_0^\times$.
\end{lemma}

\begin{remark}\label{n_non_prime_Kerdlog_remark}
On the contrary, when $n$ is not prime, this kernel can be significantly larger. For example, when $n=6$, it contains elements such as $3\e_0^2+2\e_0^3$. 
\end{remark}

The following characterization of $\Gamma_\sigma$ will be used to pinpoint the exact element in a coset that $\Gamma_\sigma$ represents.

\begin{lemma}\label{lem:sumofdi}
Write $\Gamma_\sigma = \sum_{0\leq i < n} d_i \e_0^i$, with $d_i \in \Ash$, for an element in $\bar \Lambda_0^\times$ which is a $\dsh$-preimage of $B_\sigma$. Then 
$d_{\Sigma}:=\sum_{0\leq i <n} d_i = 1$.
\end{lemma}

\begin{proof}
By Fact \ref{Faugmentation}, $B_\sigma - 1$ is in the augmentation ideal $(1-\epsilon_0)(1-\epsilon_1) \Lambda_1$. Since 
\[B_\sigma -1 =  \frac{\left( \sum d_i \e_0^i\right) \left( \sum d_i \e_1^i\right)}{\left( \sum d_i \e_0^i \e_1^i\right)} -1, \]
it lies in the augmentation ideal if and only if the difference 
\[ \left( \sum d_i \e_0^i\right) \left( \sum d_i \e_1^i\right) - \left( \sum d_i \e_0^i \e_1^i\right) \]
does. But the augmentation of the latter is precisely $(d_\Sigma^2-d_\Sigma) = d_\Sigma(d_\Sigma-1)$. As $\Gamma_\sigma$ is invertible, $d_\Sigma$ must also be invertible, hence $d_\Sigma=1$.
\end{proof}

We are now ready to prove Proposition \ref{prop:BfromPsi}.

\begin{proof}
Consider $\Psi_\sigma = \sum_{0\leq i < n} c_i \e_0^i \dlog \e_0$, with $c_i \in \Ash$. 
By Fact \ref{fact:Gammaunique}, \cite[Theorem 9]{Anderson}, 
$B_\sigma$ is uniquely determined by $\Gamma_\sigma$ in an explicit way, as $B_\sigma =\dsh(\Gamma_\sigma)$. Hence it suffices to show that $\Gamma_\sigma$ is determined by $\Psi_\sigma$ in a way unique modulo the kernel of $\dsh$.

Corollary \ref{cor:GammaPsi} gives that 
\[\dlog \Gamma_\sigma = \alpha \dlog\e_0 + \sum_{0\leq i < n} c_i \e_0^i \dlog \e_0, \]
for some $\alpha\in \Ash$. Note that the kernel $\LL$ (cf. Fact \ref{fact:Gammaunique}) of 
$\dsh:\bar\Lambda_0^\times \to \bar \Lambda_1^\times$ maps under $\dlog$ to the kernel $\Ker(\dsh_\Omega)$ of the map 
\[\dsh_{\Omega}: \Omega(\bar \Lambda_0) \to \Omega(\bar \Lambda_1), \]
which is given by $\dlog(\dsh) $, i.e.
\[\dsh_{\Omega}  (\sum a_i\e_0^i \dlog\e_0 )  = \sum a_i \e_0^i(1 - \e_1^i ) \dlog\e_0  +  \sum a_i \e_1^i(1 - \e_0^i ) \dlog\e_1.   \]
By \cite[8.5.1]{Anderson}, $\Ker(\dsh_\Omega)$ is precisely $\Ash \dlog\e_0$. When $n$ is prime, 
$\dlog: \LL \to \Ker(\dsh_\Omega)$ is an isomorphism by \cite[8.3.1]{Anderson}, which determines $\LL$.
Hence the ambiguity that $\alpha$ introduces is irrelevant for the computation of $B_\sigma$.

The remaining obstruction to reconstructing $\Gamma_\sigma$, and therefore $B_\sigma$, is the kernel of $\dlog:\bar\Lambda_0^\times \to \Omega(\bar\Lambda_0)$.

By Lemma \ref{lemma:kerdlog}, when $n $ is prime, the kernel of $\dlog$ is $\Ash \simeq \Fpbar^\times \subset \Lambda_0^\times$. Suppose $a$ lies in this kernel; this means that $\dlog(a\Gamma_\sigma) = \dlog(\Gamma_\sigma)$. On the other hand, $\dsh (a \Gamma_\sigma ) = a \dsh(\Gamma_\sigma) = a B_\sigma $, thus $a$ could introduce an ambiguity. 

Nonetheless, this ambiguity can be eliminated using Lemma \ref{lem:sumofdi}, which asserts that the sum of the coefficients of $\Gamma_\sigma $ is fixed and equals one. Hence the sum of the coefficients of  $a \Gamma_\sigma$, for $a \in \Fpbar^\times$, must be $a$.
By Lemma \ref{lem:sumofdi}, this implies that $a\Gamma_\sigma$ is not a preimage of $B_\sigma$ unless $a=1$.

In conclusion, when $n$ is prime, $\dlog \Gamma_\sigma$ uniquely determines $\Gamma_\sigma$ and therefore $B_\sigma$.
\end{proof}

In theory, by Proposition \ref{prop:BfromPsi}, the coefficients $c_i$ of $\Psi_\sigma$ studied in Section \ref{Skappa}
uniquely and explicitly determine the coefficients of $B_\sigma$, and thus the action of $\Gal{K}$ on $H_1(U,Y; A)$. 
We carry out this computation explicitly when $n=3$ in the following subsection.

\subsection{The case $n=3$}

Consider the smallest example, i.e., that of $n=3$.  Write
\[\Psi_\sigma = (c_1\e_0+ c_2\e_0^2) \dlog\e_0,\] for $c_1,c_2 \in \Fpbar \simeq \Ash$. Write
\[\Gamma_\sigma = d_0 +d_1\e_0+d_2\e_0^2,\] with $d_i \in \Fpbar$ such that $d_0+d_1+d_2 =1$. 
To determine the $d_i$'s in terms of the $c_i$'s, it is easier to work with the nilpotent variable $y=\e_0-1$ instead of $\e_0$
and use the basis $dy = \e_0 \dlog\e_0$ of $\Omega(\Lambda_0)$.

Indeed, $\Gamma_\sigma = 1+(d_1-d_2)y+d_2 y^2$, and
\[ \Psi_\sigma =  (c_1+c_2  + c_2 y) dy.  \]

By Fact \ref{fact:dlogPsi}, $\dlog\Gamma_\sigma$ agrees with $\Psi_\sigma$ modulo terms in $ \bar\F_3 \dlog \e_0 =  \bar\F_3 (y+1)^2 dy$. Therefore, for some $\alpha \in \bar\F_3$, one sees that
\[\dlog \Gamma_\sigma = \Psi_\sigma + \alpha (y+1)^2 dy, \]
which yields the equalities
\begin{align*}
d_2 - d_1 &= c +\alpha \\
-d_2 &= (c+\alpha)^2 + c_2 -\alpha \\
0 &= d_2(c+ \alpha) + (c+\alpha)(c_2-\alpha)+\alpha,
\end{align*}
where $c= c_1 +c_2$.
In particular, $\alpha$ must be a solution of the polynomial equation
\[\alpha^3 - \alpha + c^3 =0. \]
For an arbitrary choice of solution $\alpha$, the coefficients of $\Gamma_\sigma$ are
\begin{align*}
d_1 & = c_1 - \alpha -(c+\alpha)^2,\\
d_2 & =-c_2 +\alpha -(c+\alpha)^2.
\end{align*}

Note that the inverse of $\Gamma_\sigma$ expressed in the original $\e_0$-basis is
\[
\Gamma_\sigma^{-1} = (1 + d_1 +d_2 +(d_2-d_1)^2  ) + ((d_2-d_1)^2 - d_1 )\e_0 + ((d_2-d_1)^2 - d_2 )\e_0^2. \]
In terms of the $c$'s and $\alpha$, this becomes
\[ 
\Gamma_\sigma^{-1} = (1+c_1-c_2 - (c+\alpha)^2) + (c_2 + c - (c+\alpha)^2)\e_0 + ( c_2 - \alpha -(c+\alpha)^2)\e_0^2.
\]

Now $B_\sigma=\dsh(\Gamma_\sigma)$ can be computed. 

\begin{lemma}\label{p=3formulaB}
Suppose $\Psi_\sigma = (c_1 \e_0 + c_2 \e_0^2) \dlog \e_0$, and let $b_{i,j}$ be the coefficient of $\e_0^i\e_1^j$
in $B_\sigma$. Then
\begin{equation}\label{eq:BsfromCs}
\begin{aligned}
b_{0,0} &= 1+ c_2-c_1 - (c_2-c_1)^2\\
b_{0,1} & = c_1-(c_2- c_1)^2 \\
b_{1,1} & = -c_1-(c_2- c_1)^2.
\end{aligned}
\end{equation}
The rest of the coefficients are determined by symmetry $b_{i,j}=b_{j,i}$
and the fact that $b_{0,0}+b_{0,1}+b_{0,2} = 1$, $b_{1,0}+b_{1,1}+b_{1,2} = 0$, and $b_{2,0}+b_{2,1}+b_{2,2} = 0$.
\end{lemma}
\begin{remark}

From the proof of Corollary \ref{Cformulapsin=p}, if $i=1, \ldots, \frac{p-1}{2}$, then $\kappa(1-\e^{-i}) = c_{p-i} - i c_0$. 
By Proposition \ref{Pkappac}, $c_2= \kappa(1-\e^{-2})$. 
Rearranging terms gives $c_0 = c_2-c_1$, and it follows that Lemma \ref{p=3formulaB} completes the proof of Theorem \ref{thm:action3}.
\end{remark}

\section{Homology of the affine and projective Fermat curve} \label{Shomnonrel}

In this section, we determine the Galois module structure of the homology of the projective Fermat curve $X$ 
and its affine open $U=X-Y$ with coefficients in $\Aa = \ZZ/n$ for all $n \geq 3$.

\subsection{Homology of the affine curve}

We first determine the Galois module structure of $H_1(U)$, where $H_i(U)$ abbreviates $H_i(U;A)$, and more generally, all homology groups will be taken with coefficients in $A$.

The closed subset $Y \subset U$ given by $xy=0$ consists of the $2n$ points 
\[R_i=[\zeta^i: 0:1], \  Q_i=[0: \zeta^i:1].\]
Thus, $H_0(Y) \simeq \Lambda_0 \oplus \Lambda_0 $ is generated by $\e_0 \oplus 0$ and 
$0 \oplus \e_1$.
The first copy indexes the points $R_i$ 
and the second copy indexes the points $Q_i$.
The homomorphism $H_0(Y) \to H_0(U) \simeq \Aa$ sends both $\e_0 \oplus 0$ and $0 \oplus \e_1$ to $1$.

Note that $H_0(Y)$ is an $\Lambda_1$-module via $\e_0 \mapsto \e_0 \oplus 1$ 
and $\e_1 \mapsto 1 \oplus \e_1$.
The boundary map 
$\delta: H_1(U, Y) \to H_0(Y)$
is a $\Lambda_1$-module map given by 
\begin{align}\label{eq:deltabeta}
\beta \mapsto 1 \oplus 0 - 0 \oplus 1.
\end{align}

\begin{lemma} \label{Llong1}
There is an exact sequence of Galois modules
\begin{equation} \label{Eexactrelhom}
0 \to H_1(U) \to H_1(U, Y) \stackrel{\delta}{\to} H_0(Y) \to H_0(U) \to 0.
\end{equation}
The first Betti number of $U$ is $(n-1)^2$.
\end{lemma}

\begin{proof}
This follows from the long exact sequence for relative homology, using the facts
that $H_1(Y)=0$ and $H_0(U, Y)=0$.
The Betti number is the $\Aa$-rank of $H_1(U)$; note that $H_1(U, Y)$, $H_0(Y)$, and $H_0(U)$ are all free $\Aa$-modules, hence
\[{\rm rank}(H_1(U)) = {\rm rank}(H_1(U, Y)) - {\rm rank}(H_0(Y)) + {\rm rank}(H_0(U)).\]
So, the rank of $H_1(U) $ is $ n^2 - 2n + 1 = (n-1)^2.$
\end{proof}

An element $W \in \Lambda_1$ will be written as $W=\sum_{0 \leq i,j \leq n-1} a_{ij} \e_0^i \e_1^j$.

\begin{proposition} \label{Phomaffine}
Let  $W=\sum_{0 \leq i,j \leq n-1} a_{ij} \e_0^i \e_1^j$ be an element of $ \Lambda_1 $, and consider the corresponding element $W \beta$ of $H_1(U, Y)$.
Then $W \beta $ restricts to $ H_1(U)$ if and only if for each $0 \leq j \leq n-1$,
$\sum_{i = 0}^{n-1} a_{ij}=0$, and for each  $0 \leq i \leq n-1$, $\sum_{j = 0}^{n-1} a_{ij}=0$. 
\end{proposition}

\begin{proof}
By Lemma \ref{Llong1}, $W \beta \in H_1(U)$ if and only if $W \beta \in {\rm ker}(\delta)$.
Note that, by \eqref{eq:deltabeta}, 
\[\delta(\e_0 \beta) = (\e_0 \oplus 1)(1 \oplus 0 - 0 \oplus 1)=\e_0 \oplus 0 - 0 \oplus 1,\]
and similarly, 
\[\delta(\e_1 \beta) = (1 \oplus \e_1)(1 \oplus 0 - 0 \oplus 1)= 1 \oplus 0 - 0 \oplus \e_1.\]
Thus 
\begin{align*}
\delta(W \beta) & = \sum_{0 \leq i,j \leq n-1} a_{ij} \delta(\e_0^i \e_1^j \beta)\\
& = \sum a_{ij} (\e_0^i \oplus 1)(1 \oplus \e_1^j)(1 \oplus 0 - 0 \oplus 1)\\
& = \sum a_{ij} (\e_0^i \oplus 0 - 0 \oplus \e_1^j).
\end{align*}
So $W \beta \in {\rm ker}(\delta)$ if and only if the rows and columns of $W$ sum to zero.
\end{proof}

\subsection{Homology of the projective curve}

We next determine the Galois module structure of $H_1(X)$, which has rank $2g=n^2-3n+2$.

\begin{proposition} \label{PprojH1}
\begin{enumerate}
\item \label{H1X_exact_sequence}
There is an exact sequence of Galois modules and $A$-modules:
\[0 \to H_2(X) \to H_2(X, U) \stackrel{D}{\to} H_1(U) \to H_1(X) \to 0.\]
\item \label{D_computation}
The image of $D$ is ${\rm Stab}(\e_0 \e_1)$ where ${\rm Stab}(\e_0 \e_1)$ consists of $W \beta \in H_1(U)$ which are invariant under 
$\e_0 \e_1$, i.e., for which $a_{i+1, j+1} = a_{ij}$, where indices are taken modulo $n$ when necessary.
\item \label{H1X=H1U/stab}
As a Galois module and $A$-module, $H_1(X) =  H_1(U)/{\rm Stab}(\e_0 \e_1)$.
\end{enumerate}
\end{proposition}

\begin{proof}

\eqref{H1X_exact_sequence}: The long exact sequence in homology of the pair $(X,U)$ implies that the sequence \[\cdots \to H_2(U)  \to H_2(X) \to H_2(X, U) \stackrel{D}{\to} H_1(U) \to H_1(X) \to H_1(X,U) \to \cdots \] is exact. Since $U$ is affine, $H_2(U) = 0$. It thus suffices to show $H_1(X,U) = 0$. This follows from the fact that $H_1(X,U)$ is isomorphic as an abelian group to the singular homology of $(X(\C),U(\C))$, where $X(\C)$ and $U(\C)$ are given the analytic topology.

Here is an alternative proof that $H_1(X,U) = 0$ which does not use the analytic topology and which will be useful for part (2).
It follows from \cite[\S 4 Theorem 1]{Anderson} that Anderson's \'etale homology with coefficients in $\Aa$ \cite[\S 2]{Anderson} is naturally isomorphic to the homology of the \'etale homotopy type in the sense of Friedlander \cite{Friedlander82}. It follows from Voevodsky's purity theorem \cite[Theorem 2.23]{Morel_Voevodsky99} and the factorization of the \'etale homotopy type through $\mathbb{A}^1$ algebraic topology \cite{Isaksen04} that there is a natural isomorphism $H_i(X,U) \cong \tilde{H}_i(\vee_{(X-U)(\overline{K})} \PP_{\overline{K}})$, where $\tilde{H}_i$ denotes reduced homology. (For this, it is necessary to observe that the proof of  \cite[Theorem 2.23]{Morel_Voevodsky99} goes through with the \'etale topology replacing the Nisnevich topology.) 
Thus $$\tilde{H}_i (\PP_{\overline{K}}; \Aa) =\begin{cases} \Aa(1) &\mbox{if } i =2 \\ 
0 & \mbox{otherwise,}\end{cases} $$ where $ \Aa(1) $ denotes the module $\Aa=\Z/n$ with action given by the cyclotomic character. Over $K$, $\Aa(1) =\Aa$. It follows that 
$$ H_1(X,U) =\tilde{H}_1(\vee_{(X-U)(\overline{K})} \PP_{\overline{K}}; \Aa) = \oplus_{(X-U)(\overline{K})}  \tilde{H}_i( \PP_{\overline{K}}; \Aa) = 0.$$ As a third alternative, one can see that $H_1(X,U) = 0$ using \cite[VI Theorem 5.1]{Milne80} and a universal coefficients argument to change information about cohomology to information about homology.

\eqref{D_computation} As above, 
$$H_2(X,U) \cong \tilde{H}_1(\vee_{(X-U)(\overline{K})} \PP_{\overline{K}}) = \oplus_{(X-U)(\overline{K})}  \tilde{H}_i( \PP_{\overline{K}}) = \oplus_{(X-U)(\overline{K})} \Aa(1).$$ For $\eta \in (X-U)(\overline{K})$, let $\eta$ also represent the corresponding basis element of $\oplus_{(X-U)(\overline{K})} \Aa(1)$. 
Then $D(\eta)$ is represented by a small loop around $\eta$.

Note that the coordinates of $\eta$ are $[\epsilon: -\epsilon:0]$ for some $n$th root of unity $\epsilon$. In particular, $\eta$ is fixed by $\e_0 \e_1$. The loop $D(\eta)$ is therefore also fixed by $\e_0 \e_1$ because $\e_0 \e_1$ preserves orientation. 

Consider the subset ${\rm Stab}(\e_0 \e_1)$ of elements of $H_1(U)$ fixed by $\e_0 \e_1$. Then ${\rm Stab}(\e_0 \e_1)$ contains the image of $D$. 
In fact, ${\rm Stab}(\e_0 \e_1)={\rm Image}(D)$. To see this, it suffices to show that both ${\rm Stab}(\e_0 \e_1)$ and ${\rm Image}(D)$ are isomorphic to $\Aa^{n-1}$. 

By \eqref{H1X_exact_sequence}, one sees that ${\rm Image}(D)$ is isomorphic to the quotient of $H_2(X, U)$ by the image of $H_2(X) \to H_2(X, U)$. Since $X$ is a smooth proper curve, $H_2(X) \cong \Aa(1)$ and $H_2(X, U) \cong \oplus_{(X-U)(\overline{K})} \Aa(1)$. The map $H_2(X) \to H_2(X,U)$ can be described as the map that sends the basis element of $\Aa(1)$ to the diagonal element  $\oplus_{(X-U)(\overline{K})} 1$. It follows that ${\rm Image}(D) \simeq \Aa^{n-1}$ as claimed. 

Now $\e_0$ and $\e_1$ act on $H_1(U, Y)$ via multiplication.  Note that these actions have the effect of 
shifting the columns or rows of $W$ and thus stabilize $H_1(U)$. The stabilizer of $\e_0 \e_1$ is isomorphic to $\Aa^{n-1}$ because an element of the stabilizer is uniquely determined by an arbitrary choice of $a_{01}, a_{02}, \ldots, a_{0n}$.

\eqref{H1X=H1U/stab} is immediate from \eqref{H1X_exact_sequence} and \eqref{D_computation}.
\end{proof}

\section{Computing Galois cohomology when $p=3$} \label{Scohom}

In this section, we explicitly compute several cohomology groups when $p=3$.
Let $e=\zeta_0$ and $f=\zeta_1$.

\subsection{Computation of $B_\sigma$}

Let $\sigma$ and $\tau$ denote the generators of $G={\rm Gal}(L/K) \simeq (\Z/3)^2$ such that 
\begin{align*}
\sigma: \sqrt[3]{\zeta} &\mapsto \zeta \sqrt[3]{\zeta}, \qquad &\tau:\sqrt[3]{\zeta} &\mapsto \sqrt[3]{\zeta}\\
\sqrt[3]{1-\zeta^{-1}} &\mapsto \sqrt[3]{1-\zeta^{-1}} \qquad &\sqrt[3]{1-\zeta^{-1}} &\mapsto \zeta\sqrt[3]{1-\zeta^{-1}}.
\end{align*}
The equality $-\zeta(1-\zeta^{-1})=1-\zeta$ shows that $(c_1)_\sigma = 0$,  $(c_2)_\sigma = 1$ and $(c_1)_\tau = 1$, $(c_2)_\tau = 1$.
By Lemma \ref{p=3formulaB}, this implies that
\begin{equation} \label{eq:Betas}
\begin{aligned}
B_\sigma &= 1-(\ee+\ff) + (\ee^2 - \ee \ff +\ff^2) - (\ee^2\ff +\ee \ff^2)\\
& = 1 - (\ee+\ff)(1-\ee)(1-\ff),\\
B_\tau &= 1+(\ee+\ff) - (\ee^2+ \ee \ff+\ff^2) +\ee^2 \ff^2.
\end{aligned}
\end{equation}

\subsubsection{The kernel and image of $B$}

Let $G=\langle \sigma, \tau \rangle$.  
Consider the map 
\[B: \FF_3[G] \to \Lambda_1, \ B(\sigma)=B_\sigma.\]
When $p=3$, the domain and range of $B$ both have dimension $9$.
Of course, $B$ is not surjective since its image is contained in the 6 dimensional subspace of symmetric elements.

\begin{lemma}
When $p=3$, the image of $B$ has dimension 4 and the kernel of $B$ has dimension $5$.
In particular, ${\rm Im}(B)$ consists of symmetric elements whose 2nd and 3rd rows sum to 0, i.e., elements of the form
\[a_{00} + a_{01}(\ee+\ff)+a_{02}(\ee^2+\ff^2)+a_{11}\ee\ff-(a_{01}+a_{11})(\ee^2\ff+\ee \ff^2)+(a_{01}+a_{11}-a_{02})\ee^2\ff^2;\]
and ${\rm Ker}(B)$ is determined by the relations:
\noindent $$B_{\tau^2}+B_\tau+B_1=0,$$
\noindent $$B_{\sigma^2\tau}-B_{\sigma^2}-B_\tau+B_1=0,$$
\noindent $$B_{\sigma \tau}-B_\sigma-B_\tau+B_1=0,$$
\noindent $$B_{\sigma^2 \tau^2}-B_{\sigma^2}-B_{\tau^2}+B_1=0,$$
\noindent $$B_{\sigma \tau^2}-B_\sigma-B_{\tau^2}+B_1=0.$$
\end{lemma}

\begin{proof}
Magma computation using \eqref{eq:Betas}.
\end{proof}

\subsection{Cohomology of ${\rm Gal}(L/K)$}

Since $\sigma$ has order $3$, by \cite[Example 1.1.4, Exercise 1.2.2]{Brown},
the projective resolution of $\Z$ as a $\Z[\langle \sigma \rangle]$-module is 
\[\Z[G] \stackrel{1-\sigma}{\leftarrow} \Z[G] \stackrel{1+\sigma +\sigma^2}{\leftarrow}  \Z[G]  \stackrel{1-\sigma}{\leftarrow} \cdots \]
where $G={\rm Gal}(L/K) = \langle \sigma, \tau \mid \sigma^3=\tau^3=[\sigma,\tau]=1 \rangle \simeq (\Z/3)^2$. 
By \cite[Proposition V.1.1]{Brown}, 
the total complex associated to the following double complex is a projective resolution of $\Z$ as a $\Z[G]$-module:
\[ \xymatrix{
\Z[G] \ar[d]^{1+\tau+\tau^2} & \Z[G] \ar[l]_{1-\sigma} \ar[d]^{1+\tau+\tau^2} & \Z[G] \ar[l]_{1+\sigma+\sigma^2}\ar[d]^{1+\tau+\tau^2}  \ar@{<-}[r]^{1-\sigma} &\\
\Z[G] \ar[d]^{1-\tau} &\Z[G] \ar[l]_{-(1-\sigma)} \ar[d]^{1-\tau} & \Z[G] \ar[l]_{-(1+\sigma+\sigma^2)} \ar[d]^{1-\tau}  \ar@{<-}[r]^{1-\sigma} &\\
\Z[G]  & \Z[G] \ar[l]_{1-\sigma} & \Z[G] \ar[l]_{1+\sigma+\sigma^2}  \ar@{<-}[r]^{1-\sigma} &
 }\]
Therefore, to compute $H^1(G,M)$, one can compute the cohomology of the complex
\[ \xymatrix@R=0pt@C+15pt{
	& 		& M \\
	& M \ar[ru]^{1+\sigma+\sigma^2} \ar[rd]^{1-\tau}		& \oplus \\
M\ar[ru]^{1-\sigma} \ar[rd]_{1-\tau}	& \oplus	& M\\
	& M 	\ar[ru]^{-(1-\sigma)} \ar[rd]_{1+\tau+\tau^2}	& \oplus \\
	&	 	& M.
}\]

Given $h \in \FF_3[G]$, let ${\rm Ann}_M(h)=\{m \in M \mid hm=0\}$.  Let $M=\Lambda_1$.

\begin{lemma} \label{Lannih}
Let $M=\Lambda_1$ with $\ee=\epsilon_0$ and $\ff = \epsilon_1$.
\begin{enumerate}
\item ${\rm Ann}_M(1+\tau + \tau^2)=M$.
\item ${\rm Ann}_M(1+\sigma + \sigma^2)=(1-\ee, 1-\ff)$ consists of all $m = \sum m_{ij} \ee^i \ff^j$ such that $\sum m_{ij} = 0$.
\item ${\rm Ann}_M(1-\sigma)=(1+\ee+\ee^2, 1+\ff+\ff^2)$.
\item ${\rm Ann}_M(1-\tau)=(\ee-\ff, 1+\ff+\ff^2)$. 
\end{enumerate}
\end{lemma}

\begin{proof}
\begin{enumerate}
\item Every $m \in M$ is in the annihilator of $1+\tau+\tau^2$ because $1+B_\tau+B_{\tau^2}$ equals
\[(\ee +\ff) - (\ee^2 +\ee\ff+\ff^2) + \ee^2 \ff^2 + (\ee +\ff)^2 + (\ee^2 +\ee \ff+\ff^2)^2 + \ee \ff \]
\[+(\ee +\ff)(\ee^2 +\ee \ff+\ff^2) - (\ee +\ff)\ee^2 \ff^2 + (\ee^2 +\ee \ff+\ff^2)\ee^2 \ff^2,\]
which is zero.

\item Note that $B_\sigma = 1- (\ee + \ff)(1-\ee)(1-\ff)$,
which gives that
\begin{align*}
1 + B_\sigma+B_{\sigma^2} &= (\ee^2-\ee \ff+\ff^2)(1+\ee+\ee^2) (1+\ff +\ff^2) \\
&= (1+\ee+\ee^2) (1+\ff +\ff^2).
\end{align*}
Note that $(1+B_\sigma+B_{\sigma^2})\ee^i \ff^j = (1+B_\sigma+B_{\sigma^2}),$ so for $m=\sum_{i,j}m_{i,j}\ee^i \ff^j$, 
\[(1+B_\sigma+B_{\sigma^2}) m = (1+B_\sigma+B_{\sigma^2}) (\sum_{i,j} m_{i,j}). \]
Thus ${\rm Ann}_M(1+B_\sigma+B_{\sigma^2})$ consists of $m \in M$ whose entries sum to $0$.

\item
Note that $(1-\sigma)m=0$ if and only if $\sigma m=m$ which in turn simplifies to $(\ee+\ff)(1-\ee)(1-\ff)m=0$.
Thus ${\rm Ann}_M(1-\sigma)$ is generated by the annihilators of 
$1-\ee$ and $1-\ff$ and $\ee+\ff$, which are
$1+\ee+\ee^2$ and $1+\ff+\ff^2$ and 0.

\item A Magma calculation using that $1 - B_\tau=-(\ee+\ff)+(\ee^2+\ff^2)+\ee \ff-\ee^2\ff^2$.
\end{enumerate}
\end{proof}

\subsection{Preliminary calculations}

Consider the maps $X: M \rightarrow M^2$, $Y:M^2 \rightarrow M^3$, and $Z:M^3 \rightarrow M^4$. 
The goal is to compute $\ker(Y)/\im(X)$ and $\ker(Z)/\im(Y)$. 

After choosing a basis for $M$, the maps $X$, $Y$, and $Z$ can be written in matrix form.  The basis of $M$ chosen here is 
\[1, \ff, \ff^2, \ee, \ee\ff, \ee\ff^2, \ee^2, \ee^2\ff, \ee^2\ff^2.\]

By Lemma \ref{Lannih}, 
all of the entries of the matrix $V$ for the map $\Nm(\tau):M \rightarrow M$ are $0$.
All of the entries of the matrix $U$ for the map $\Nm(\sigma):M \rightarrow M$ are $1$
since $\Nm(\sigma)$ acts on each element of $M$ by summing its coefficients.

Let $S$ be the matrix for the map $1-B_\sigma: M \rightarrow M$. 
Let $T$ be the matrix for the map $1-B_\tau: M \rightarrow M$.   
Here are the matrices $S$ and $T$:

\[ S=\left[ \begin{array}{ccc ccc ccc}
0 &1 &2&1&1&1&2&1&0\\
2 &0 &1&1&1&1&0&2&1\\
1 &2 &0&1&1&1&1&0&2\\

2 &1 &0&0&1&2&1&1&1\\
0 &2 &1&2&0&1&1&1&1\\
1 &0 &2&1&2&0&1&1&1\\

1 &1 &1 &2&1&0&0&1&2\\
1 &1 &1 &0&2&1&2&0&1\\
1 &1 &1 &1&0&2&1&2&0\\
\end{array} \right];\] 
and 
\[ T=\left[ \begin{array}{ccc ccc ccc}
0 &2 &1&2&1&0&1&0&2\\
1 &0 &2&0&2&1&2&1&0\\
2 &1 &0&1&0&2&0&2&1\\

1 &0 &2&0&2&1&2&1&0\\
2 &1 &0&1&0&2&0&2&1\\
0 &2 &1&2&1&0&1&0&2\\

2 &1 &0 &1&0&2&0&2&1\\
0 &2 &1 &2&1&0&1&0&2\\
1 &0 &2&0 &2&1 &2 &1& 0\\
\end{array} \right].\] 

The block matrices for $X$, $Y$, and $Z$ are given as follows:
\[ X=\left[ \begin{array}{cc}
S & T
\end{array} \right];\] 

\[ Y=\left[ \begin{array}{ccc}
U & T & 0 \\
0 & -S & V=0
\end{array} \right];\] 

\[ Z=\left[ \begin{array}{cccc}
S & T & 0 & 0  \\
0 & -U & V=0 & 0 \\
0 & 0 & S & T
\end{array} \right].\] 

\subsection{Calculation of $H^1({\rm Gal}(L/K), M)$}

The plan is to compute the cohomology of the complex:
\[ \xymatrix@R=0pt@C+15pt{
	& 		& M \\
	& M \ar[ru]^{1+\sigma+\sigma^2} \ar[rd]^{1-\tau}		& \oplus \\
M\ar[ru]^{1-\sigma} \ar[rd]_{1-\tau}	& \oplus	& M\\
	& M 	\ar[ru]^{-(1-\sigma)} \ar[rd]_{1+\tau+\tau^2}	& \oplus \\
	&	 	& M.
}\]

\begin{lemma}
The kernel of $Y: M^2 \to M^3$ has dimension $13$ and a basis is:
\begin{align*}
(\ff-\ee^2\ff^2) &\oplus 0, \\
(\ee-\ee^2\ff^2) &\oplus 0, \\
(1-\ee^2\ff^2) & \oplus (\ee \ff-\ee \ff^2-\ee^2\ff+\ee^2\ff^2), \\
(\ff^2-\ee^2\ff^2) &\oplus (-\ee\ff+\ee\ff^2+\ee^2\ff-\ee^2\ff^2), \\
(\ee\ff-\ee^2\ff^2) &\oplus (-\ee\ff+\ee\ff^2+\ee^2\ff-\ee^2\ff^2), \\
(\ee\ff^2-\ee^2\ff^2) &\oplus (\ee\ff-\ee\ff^2-\ee^2\ff+\ee^2\ff^2), \\
(\ee^2-\ee^2\ff^2) &\oplus (-\ee\ff+\ee\ff^2+\ee^2\ff-\ee^2\ff^2), \\
(\ee^2\ff-\ee^2\ff^2) &\oplus (\ee\ff-\ee\ff^2-\ee^2\ff+\ee^2\ff^2), \\
0 &\oplus (1-\ee\ff-\ee\ff^2-\ee^2\ff-\ee^2\ff^2), \\
0 &\oplus (\ff+\ee\ff+\ee^2\ff), \\
0 &\oplus (\ff^2+\ee\ff^2+\ee^2\ff^2), \\
0 &\oplus (\ee+\ee\ff+\ee\ff^2), \\
0 &\oplus (\ee^2+\ee^2\ff+\ee^2\ff^2). 
\end{align*}

\end{lemma}

\begin{proof}
Magma calculation.
\end{proof}

\begin{lemma}
The image of $X: M \to M^2$ has dimension $4$ and a basis is:
\begin{align*}
(1-\ff^2-\ee^2+\ee^2\ff^2) &\oplus (1-\ff^2-\ee\ff+\ee\ff^2-\ee^2+\ee^2\ff), \\
(\ff-\ff^2-\ee^2+\ee^2\ff) &\oplus (1-\ff+\ff^2-\ee-\ee\ff^2-\ee^2+\ee^2\ff^2), \\
(\ee-\ee\ff^2-\ee^2+\ee^2\ff^2) &\oplus (\ff-\ff^2+\ee-\ee\ff-\ee^2+\ee^2\ff^2), \\
(\ee\ff-\ee\ff^2-\ee^2\ff+\ee^2\ff^2) &\oplus (-1+\ff+\ee-\ee\ff^2-\ee^2\ff+\ee^2\ff^2).
\end{align*}
\end{lemma}

\begin{proof}
Magma calculation.
\end{proof}

Note that the image of $X$ is contained in the kernel of $Y$.

\begin{proposition}
The dimension of $H^1({\rm Gal}(L/K), M)$ is 9 and a basis is:
\begin{align*}
(\ff^2-\ee^2)   &\oplus 0, \\
(\ee\ff^2-\ff\ee^2)   &\oplus 0, \\
(\ee^2+\ee^2\ff+\ee^2\ff^2)   &\oplus 0, \\ 
(\ee^2\ff-\ee^2\ff^2)   &\oplus (\ee\ff-\ee\ff^2-\ee^2\ff+\ee^2\ff^2), \\ 
0   &\oplus (1-\ee\ff-\ee\ff^2-\ee^2\ff-\ee^2\ff^2), \\ 
0   &\oplus (\ff+\ee\ff+\ee^2\ff), \\ 
0   &\oplus (\ff^2+\ee\ff^2+\ee^2\ff^2), \\ 
0   &\oplus (\ee+\ee\ff+\ee\ff^2), \\ 
0  &\oplus (\ee^2+\ee^2\ff+\ee^2\ff^2).  
\end{align*}
\end{proposition}

\begin{proof}
The quotient $H^1({\rm Gal}(L/K), M)=\ker(Y)/{\rm im}(X)$ can be computed using the complement function in Magma.
\end{proof}

\subsection{Calculation of $H^2({\rm Gal}(L/K), M)$}

In this section, we compute the kernel of $Z:M^3 \to M^4$ modulo the image of $Y:M^2 \to M^3$.

\begin{proposition}
The dimension of $H^2({\rm Gal}(L/K), M)$ is 13 and a basis is:
\[\begin{array}{rcccl}
(\ff+\ee\ff+\ee^2\ff) & \oplus & 0  & \oplus &  0, \\
(\ff^2+ \ee\ff^2+ \ee^2\ff^2) & \oplus  & 0 & \oplus &  0, \\
(\ee+\ee\ff+\ee\ff^2) & \oplus & 0 & \oplus  & 0,\\
(\ee^2+\ee^2\ff+\ee^2\ff^2) & \oplus & 0 &  \oplus & 0,\\
0 &\oplus & (\ff^2-\ee^2\ff^2) & \oplus & 0, \\
0 &\oplus & (\ee\ff^2-\ee^2\ff^2) & \oplus & 0,\\
0 &\oplus & (\ee^2-\ee^2\ff^2) & \oplus & 0,\\
0 &\oplus & (\ee^2\ff -\ee^2\ff^2) & \oplus & 0,\\
0 & \oplus & 0 & \oplus &  (1-\ee\ff-\ee\ff^2 -\ee^2\ff-\ee^2\ff^2),\\
0 & \oplus  & 0 & \oplus & (\ff+\ee\ff+\ee^2\ff),\\
0 & \oplus & 0 & \oplus & (\ff^2+ \ee\ff^2+ \ee^2\ff^2),\\ 
0 & \oplus & 0 & \oplus & (\ee+\ee\ff+\ee\ff^2),\\
0 & \oplus & 0 & \oplus & (\ee^2+\ee^2\ff+\ee^2\ff^2).
\end{array}\]
\end{proposition}

\subsection{First and second cohomology with coefficients in $H_1(U)$}

Recall that $V=H_1(U)$ has dimension $(p-1)^2$.
If $p=3$, then $V$ is a 4-dimensional subspace of $M$ and a basis for $V$ is:

\[ v_1=\left[ \begin{array}{ccc}

1 & 0 & -1 \\
0 & 0 & 0 \\
-1 & 0 & 1 \\
\end{array} \right];\] 

\[ v_2=\left[ \begin{array}{ccc}

0 & 0 & 0 \\
1 & 0 & -1 \\
-1 & 0 & 1 \\
\end{array} \right];\] 

\[ v_3=\left[ \begin{array}{ccc}

0 & 1 & -1 \\
0 & 0 & 0 \\
0 & -1 & 1 \\
\end{array} \right];\] 

\[ v_4=\left[ \begin{array}{ccc}

0 & 0 & 0 \\
0 & 1 & -1 \\
0 & -1 & 1 \\
\end{array} \right].\] 

Let $X_1$, $Y_1$, and $Z_1$ be the restriction of $X$, $Y$, and $Z$ respectively.
Similarly, let $S_1=1-B_\sigma$, $T_1=1-B_\tau$, $U_1=\Nm(\sigma)$, and $V_1=\Nm(\tau)$ 
be the restrictions of $S$, $T$, $U$, and $V$ respectively.
Then $T_1$, $U_1$, and $V_1$ are each the $4 \times 4$ zero matrix and
\[ S_1=\left[ \begin{array}{cccc}

-1 & -1 & -1 & -1 \\
1 & 1 & 1  & 1\\
 1 & 1 & 1 & 1 \\
 -1 & -1 & -1 & -1\\
\end{array} \right].\] 

Then
\[ X_1=\left[ \begin{array}{cc}
S_1 & T_1=0
\end{array} \right];\] 

\[ Y_1=\left[ \begin{array}{ccc}
U_1=0 & T_1=0& 0 \\
0 & -S_1 & V_1=0
\end{array} \right];\] 

\[ Z_1=\left[ \begin{array}{cccc}
S_1 & T_1=0 & 0 & 0  \\
0 & -U_1=0 & V_1=0 & 0 \\
0 & 0 & S_1 & T_1=0
\end{array} \right].\] 

\begin{proposition}
The dimension of $H^1(G,H_1(U))$ is $6$ and a basis is
\[v_2 \oplus 0, v_3 \oplus 0, v_4 \oplus 0, 0 \oplus (v_1-v_4),  0 \oplus (v_2-v_4), 0 \oplus (v_3-v_4).\]

The dimension of $H^2(G,H_1(U))$ is 9 and a basis is:
\[\begin{array}{rcccl}
(v_1-v_4) & \oplus & 0 & \oplus & 0,\\
(v_2+v_4) & \oplus & 0  & \oplus & 0,\\
(v_3+v_4) & \oplus &  0  & \oplus &  0,\\
0 & \oplus &  v_2 &  \oplus & 0,\\
0 & \oplus &  v_3 &  \oplus & 0,\\
0 & \oplus &  v_4  & \oplus & 0,\\
0 & \oplus &  0  & \oplus & (v_1-v_4),\\
0 & \oplus & 0 &  \oplus &  (v_2+v_4),\\
0 & \oplus &  0  & \oplus & (v_3+v_4).
\end{array}\]
\end{proposition}

\subsection{First and second cohomology with coefficients in $H_1(U) \wedge H_1(U)$ }

The vector space $W=H_1(U) \wedge H_1(U)$ has dimension ${(p-1)^2 \choose 2}=6$. 

\begin{proposition} \label{Pwedge}
Let $W=H_1(U) \wedge H_1(U)$.  
Then $H^1(G, W)=W^2$ (with dimension 12) 
and $H^2(G,W)=W^3$ (with dimension 18).
\end{proposition}

\begin{proof}
The map $S^\wedge: V \wedge V \to V \wedge V$ induced by $S=(1-B_{\sigma})$ is the exterior square of $S_1$. 
One computes that $S^\wedge$ is the $6 \times 6$ zero matrix.
Similarly the matrices for $T^\wedge$, $U^\wedge$ and $V^\wedge$ are zero.
Then $H^1(G, W)=W^2$ since ${\rm Im}(X^\wedge)=0$ and ${\rm Ker}(Y)=W^2$.
Also $H^2(G,W)=W^3$ since ${\rm Im}(Y^\wedge)=0$ and ${\rm Ker}(Z)=W^3$.
\end{proof}

\bibliographystyle{amsalpha}
\bibliography{biblio}

\end{document}